\documentclass[12pt,reqno,oneside]{amsproc} \title[Exponential Decay]{Exponential Decay of Solutions to a Fluid-Plate Model with Small Initial Data}

\author[M.~Bukal]{Mario Bukal}
\address{University of Zagreb Faculty of Electrical Engineering and Computing, Unska 3, 10000 Zagreb, Croatia}
\email{mario.bukal@fer.hr}

\author[I.~Kukavica]{Igor Kukavica}
\address{Department of Mathematics, University of Southern California, Los Angeles, CA 90089}
\email{kukavica@usc.edu}

\author[L.~Li]{Linfeng Li}
\address{Department of Mathematics, University of California Los Angeles, Los Angeles, CA 90095}
\email{lli265@math.ucla.edu}

\author[B.~Muha]{Boris Muha}
\address{University of Zagreb, Faculty of Science, Department of Mathematics, Bijeni\v cka cesta 30, 10000 Zagreb, Croatia}
\email{borism@math.hr}

\usepackage[normalem]{ulem}
\usepackage{enumitem}
\usepackage{bbm}
\usepackage{fancyhdr}
\usepackage{comment}
\usepackage{tikz}
\usetikzlibrary{arrows.meta}
\chardef\forshowkeys=0
\chardef\refcheck=0
\chardef\showllabel=0
\chardef\sketches=0
\chardef\showcolors=1
\usepackage{bbm,yfonts}
\usepackage{upgreek}
\usepackage{enumitem}
\usepackage{datetime}
\usepackage{esint}
\usepackage{fancyhdr}
\usepackage{comment}
\usepackage[utf8]{inputenc}
\allowdisplaybreaks

\let\pa\partial   
  
\let\eps\epsilon  

\newcommand{\N}{{\mathbb N}}    
\newcommand{\R}{{\mathbb R}}  

\newcommand{\diver}{\operatorname{div}}

\newcommand{\dist}{\operatorname{dist}}

\newcommand{\D}{\operatorname{D}}

\newcommand{\red}{\textcolor{red}}

\ifnum\forshowkeys=1

\usepackage[notcite,color]{showkeys}
\fi

\usepackage[margin=1in]{geometry}
\usepackage{amsmath, amsthm, amssymb}
\usepackage{times}
\usepackage{graphicx}
\usepackage{marginnote}
\usepackage[unicode,breaklinks=true,colorlinks=true,linkcolor=blue,urlcolor=blue,citecolor=blue]{hyperref}
\usepackage[most]{tcolorbox}

\usepackage{tikz}

\ifnum\refcheck=1
\usepackage{refcheck}
\fi

\begin{document}
%varmac (various macros)
\def\Omegaone{\text{1-channel}}
\def\ch{{\hat{c}}}
\def\bp{\bar{\partial}}
\def\bbeta{\beta}
\def\aalpha{\mu}
\def\etah{\eta}
\def\uh{u}
\def\vh{v}
\def\qh{q}
\def\ph{p}
\def\inp{{\colr IN PROGRESS}}
\def\rth{{\colr REVISED UP TO HERE}}
\def\fNS{f}
\def\fP{g}
\def\bnew{\colr {\bf }}
\def\enew{\colb {}}
\def\bold{\colu {\bf }}
\def\eold{\colb {}}
\def\inte{\int_{\Omega}}
\def\intb{\int_{\Gamma}}
\def\F{F}
\def\Q{q_{\text{osc}}}
\def\M{M}
\def\D{S}
\def\L{\mathcal{L}}
\def\E{\mathcal{E}}
\def\DD{\mathcal{D}}
\def\N{\mathcal{N}}
\def\C{\mathcal{C}}
\def\etaa{h}
\def\CommT{{R}}
\def\BndT{B}

\def\XX{X}
\def\YY{Y}
\def\ZZZ{Z}

\def\intint{\int\!\!\!\!\int}
\def\OO{\mathcal O}
\def\SS{\mathbb S}
\def\CC{\mathbb C}
\def\RR{\mathbb R}
\def\tt{\mathbb{T}^2}
\def\ZZ{\mathbb Z}
\def\HH{\mathbb H}
\def\RSZ{\mathcal R}
\def\LL{\mathcal L}
\def\SL{\LL^1}
\def\ZL{\LL^\infty}
\def\GG{\mathcal G}
\def\erf{\mathrm{Erf}}
\def\mgt#1{\textcolor{magenta}{#1}}
\def\ff{\rho}
\def\gg{G}
\def\sqrtnu{\sqrt{\nu}}
\def\ww{w}
\def\ft#1{#1_\xi}
\def\lec{\lesssim}
	\def\les{\lesssim}
\def\gec{\gtrsim}
\renewcommand*{\Re}{\ensuremath{\mathrm{{\mathbb R}e\,}}}
\renewcommand*{\Im}{\ensuremath{\mathrm{{\mathbb I}m\,}}}

\ifnum\showllabel=1
\def\llabel#1{\marginnote{\color{lightgray}\rm\small(#1)}[-0.0cm]\notag}
\else
\def\llabel#1{\notag}
\fi

\newcommand{\norm}[1]{\left\|#1\right\|}
\newcommand{\nnorm}[1]{\lVert #1\rVert}
\newcommand{\abs}[1]{\left|#1\right|}
\newcommand{\NORM}[1]{|\!|\!| #1|\!|\!|}

\newtheorem{Theorem}{Theorem}[section]
\newtheorem{Corollary}[Theorem]{Corollary}
\newtheorem{Definition}[Theorem]{Definition}
\newtheorem{Proposition}[Theorem]{Proposition}
\newtheorem{Lemma}[Theorem]{Lemma}
\newtheorem{Remark}[Theorem]{Remark}

\def\theequation{\thesection.\arabic{equation}}
\numberwithin{equation}{section}

\def\theequation{\thesection.\arabic{equation}}
\numberwithin{equation}{section}

\definecolor{mygray}{rgb}{.6,.6,.6}
\definecolor{myblue}{rgb}{9, 0, 1}
\definecolor{colorforkeys}{rgb}{1.0,0.0,0.0}

\newlength\mytemplen
\newsavebox\mytempbox
\def\weaks{\text{\,\,\,\,\,\,weakly-* in }}
\def\weak{\text{\,\,\,\,\,\,weakly in }}
\def\inn{\text{\,\,\,\,\,\,in }}
\def\cof{\mathop{\rm cof\,}\nolimits}
\def\Id{\mathop{\rm Id\,}\nolimits}
\def\Dn{\frac{\partial}{\partial N}}
\def\Dnn#1{\frac{\partial #1}{\partial N}}
\def\tdb{\tilde{b}}
\def\tda{b}
\def\qqq{u}
\def\lat{\Delta_{\bx}}
\def\biglinem{\vskip0.5truecm\par==========================\par\vskip0.5truecm}

\def\inon#1{\hbox{\ \ \ \ \ \ \ }\hbox{#1}}               
\def\onon#1{\inon{on~$#1$}}
\def\inin#1{\inon{in~$#1$}}

\def\FF{F}
\def\andand{\text{\indeq and\indeq}}
\def\ww{w(y)}
\def\ll{{\color{red}\ell}}
\def\ee{\epsilon_0}
\def\startnewsection#1#2{ \section{#1}\label{#2}\setcounter{equation}{0}}   
\def\nnewpage{ }
\def\sgn{\mathop{\rm sgn\,}\nolimits}    
\def\Tr{\mathop{\rm Tr}\nolimits}    
\def\div{\mathop{\rm div}\nolimits}
\def\curl{\mathop{\rm curl}\nolimits}
\def\dist{\mathop{\rm dist}\nolimits}  
\def\supp{\mathop{\rm supp}\nolimits}
\def\indeq{\quad{}}           
\def\period{.}                       
\def\semicolon{\,;}
\def\pa{\partial}            
\def\pt{\partial_t}                

\ifnum\showcolors=1
\def\colr{\color{red}}
\def\colc{\color{cyan}}
\def\colrr{\color{black}}
\def\colb{\color{black}}
\definecolor{colorgggg}{rgb}{0.1,0.5,0.3}
\definecolor{colorllll}{rgb}{0.0,0.7,0.0}
\definecolor{colorhhhh}{rgb}{0.3,0.75,0.4}
\definecolor{colorpppp}{rgb}{0.7,0.0,0.2}
\definecolor{coloroooo}{rgb}{0.45,0.0,0.0}
\definecolor{colorqqqq}{rgb}{0.1,0.7,0}
\def\coly{\color{lightgray}}
\def\colg{\color{colorgggg}}
\def\collg{\color{colorllll}}
\def\cole{\color{coloroooo}}
\def\coll{\color{colorqqqq}}
\def\coleo{\color{colorpppp}}
\def\colu{\color{blue}}
\def\colc{\color{colorhhhh}}
\def\colW{\colb}   
\definecolor{coloraaaa}{rgb}{0.6,0.6,0.6}
\def\colw{\color{coloraaaa}}
\else
\def\colr{\color{black}}
\def\colrr{\color{black}}
\def\colb{\color{black}}
\def\coly{\color{black}}
\def\colg{\color{black}}
\def\collg{\color{black}}
\def\cole{\color{black}}
\def\coleo{\color{black}}
\def\colu{\color{black}}
\def\colc{\color{black}}
\def\colW{\color{black}}
\def\colw{\color{black}}
\fi

\def\comma{ {\rm ,\qquad{}} }            
\def\commaone{ {\rm ,\quad{}} }          
\def\nts#1{{\color{red}\hbox{\bf ~#1~}}} 
\def\ntsf#1{\footnote{\color{colorgggg}\hbox{#1}}} 
\def\blackdot{{\color{red}{\hskip-.0truecm\rule[-1mm]{4mm}{4mm}\hskip.2truecm}}\hskip-.3truecm}
\def\bluedot{{\color{blue}{\hskip-.0truecm\rule[-1mm]{4mm}{4mm}\hskip.2truecm}}\hskip-.3truecm}
\def\purpledot{{\color{colorpppp}{\hskip-.0truecm\rule[-1mm]{4mm}{4mm}\hskip.2truecm}}\hskip-.3truecm}
\def\greendot{{\color{colorgggg}{\hskip-.0truecm\rule[-1mm]{4mm}{4mm}\hskip.2truecm}}\hskip-.3truecm}
\def\cyandot{{\color{cyan}{\hskip-.0truecm\rule[-1mm]{4mm}{4mm}\hskip.2truecm}}\hskip-.3truecm}
\def\reddot{{\color{red}{\hskip-.0truecm\rule[-1mm]{4mm}{4mm}\hskip.2truecm}}\hskip-.3truecm}

\def\tdot{{\color{green}{\hskip-.0truecm\rule[-.5mm]{3mm}{3mm}\hskip.2truecm}}\hskip-.1truecm}
\def\gdot{\greendot}
\def\bdot{\bluedot}
\def\ydot{\cyandot}
\def\rdot{\cyandot}
\def\fractext#1#2{{#1}/{#2}}
\def\ii{\hat\imath}
\def\boris#1{\textcolor{blue}{#1}}
\def\vlad#1{\textcolor{cyan}{#1}}
\def\igor#1{\text{{\textcolor{colorqqqq}{#1}}}}
\def\linfeng#1{\textbf{\textcolor{colorqqqq}{LL:#1}}}
\def\igorf#1{\footnote{\text{{\textcolor{colorqqqq}{#1}}}}}

\newcommand{\uvro}{\upvarrho}
\newcommand{\bv}{\boldsymbol v}
\newcommand{\bV}{\boldsymbol V}
\newcommand{\bu}{\boldsymbol u}
\newcommand{\bs}{\boldsymbol}
\newcommand{\bx}{x}
\newcommand{\obx}{\overline{\boldsymbol x}}
\newcommand{\bw}{w}
\newcommand{\bn}{\mathbf n}
\newcommand{\by}{\boldsymbol y}
\newcommand{\bz}{\boldsymbol z}
\newcommand{\bef}{\boldsymbol f}
\newcommand{\material}{\Omega_{\eps,h}}
\newcommand{\vol}{\operatorname{vol}}
\newcommand{\bog}{{\rm Ext}_{\diver}}

\def\AA{Y}
\newcommand{\p}{\partial}
\newcommand{\UE}{U^{\rm E}}
\newcommand{\PE}{P^{\rm E}}
\newcommand{\KP}{K_{\rm P}}
\newcommand{\uNS}{u^{\rm NS}}
\newcommand{\vNS}{v^{\rm NS}}
\newcommand{\pNS}{p^{\rm NS}}
\newcommand{\omegaNS}{\omega^{\rm NS}}
\newcommand{\uE}{u^{\rm E}}
\newcommand{\vE}{v^{\rm E}}
\newcommand{\pE}{p^{\rm E}}
\newcommand{\omegaE}{\omega^{\rm E}}
\newcommand{\ua}{u_{\rm   a}}
\newcommand{\omegaa}{\omega_{\rm   a}}
\newcommand{\ue}{u_{\rm   e}}
\newcommand{\ve}{v_{\rm   e}}
\newcommand{\omegae}{\omega_{\rm e}}
\newcommand{\omegaeic}{\omega_{{\rm e}0}}

\def\mario#1{\textcolor{cyan}{#1}}
\def\red#1{\textcolor{blue}{#1}}

\begin{abstract}
We consider a three-dimensional fluid-structure interaction problem coupling the incompressible Navier-Stokes equations in a time-dependent domain with a square-root damped plate equation, which is posed on the moving upper boundary of the fluid. 
We prove, a~priori, the exponential decay of strong solutions for initial data that are sufficiently small in a suitable Sobolev space.
The proof combines higher-order energy estimates, Stokes-type regularity bounds, and a nonlinear bootstrap scheme that closes under the smallness assumption.
\end{abstract}

\maketitle

\setcounter{tocdepth}{1}

\colb
\section{Introduction}
\label{secintro}

Fluid-structure interaction problems involving elastic structures have been studied extensively in recent years. For viscous incompressible fluids interacting with plates, shells, or beams, the existence of weak solutions has been established in a variety of settings.
Chambolle, Desjardins, Esteban, and Grandmont~\cite{CDEG} proved the existence of weak solutions for an incompressible viscous fluid coupled to a viscoelastic plate, while Grandmont~\cite{Gr1} subsequently extended this result to an elastic plate. Weak-solution theories were later developed for more general geometries and structural models. Lengeler and R\r{u}\v{z}i\v{c}ka~\cite{LR} considered a three-dimensional fluid interacting with a linearly elastic Koiter shell of general geometry. Fluid--structure models in cylindrical geometries, with elastic or viscoelastic walls in both two- and three-dimensional settings, were studied in a series of works by Muha and \v{C}ani\'c; see, for example,~\cite{MC1,MC2}. Muha and Schwarzacher~\cite{MS} treated a three-dimensional Navier--Stokes fluid interacting with a nonlinear Koiter shell and obtained, in addition, improved regularity of the shell displacement due to fluid dissipation. Trifunovi\'c and Wang~\cite{TW} considered a three-dimensional incompressible fluid coupled to nonlinear plate models encompassing the generalized Kirchhoff, von K\'arm\'an, and Berger cases. In two dimensions, Casanova, Grandmont, and Hillairet~\cite{CGH} established an existence theory for a fluid--beam system that allows contact. We refer to the works cited above and the references therein for further results on weak solutions for fluid--structure interaction problems.

Recent work on fluid-plate models has paid particular attention to no-contact results and geometric
nondegeneracy, since they are essential to the physical relevance of these models in certain biological and engineering applications.  Breit and Roy \cite{BR} proved a conditional no-contact result for a two-dimensional
compressible fluid coupled to an elastic plate. In the three-dimensional incompressible setting \cite{BKLM1}, very
close to the present paper, the authors showed uniform separation
of the plate from the rigid bottom under the natural energy bounds and additional regularity
assumptions. By ruling out the contact under a broad condition on the initial data \cite{BKLM2}, the authors then established the global existence of weak solutions. 
A particular aspect of fluid-plate coupling is described by reduced models, which are favorable in engineering applications. Bukal and Muha \cite{BM} provided a rigorous link between a two-dimensional fluid-structure system and a nonlinear sixth-order thin-film equation as an elastohydrodynamic model that is more tractable both analytically and numerically. 

Compared to weak solutions, global strong solutions for fluid-elastic structure systems are more delicate to construct, and their existence has typically been established either locally in time for general data or globally in time under suitable smallness assumptions. In~\cite{Be}, Beir\~ao da~Veiga proved the existence of strong solutions for a coupled fluid-structure evolution problem. In~\cite{Le1,Le2}, Lequeurre established local-in-time existence of strong solutions for arbitrary initial data and global-in-time existence for small initial data for closely related fluid-structure interaction systems, in both two- and three-dimensional settings. We note, however, that in these global strong-solution results the smallness threshold depends on the prescribed time interval.
In two dimensions, Grandmont and Hillairet established in \cite{GH} global strong solutions for a viscoelastic beam-fluid interaction system. In three dimensions, Djebour and Takahashi \cite{DjTa} obtained the local in time existence and uniqueness of strong solutions to a closely related fluid structure interaction problem, but with Navier boundary conditions. Moreover, they show that these solutions are global in time provided small initial data and assuming the presence of frictions in the boundary conditions.

For a related fluid--structure model of a damped elastic body interacting with an incompressible fluid, Ignatova, Kukavica, Lasiecka, and Tuffaha first established a priori estimates for local existence and uniqueness~\cite{IKLT1}, and later proved small-data global existence and exponential decay~\cite{IKLT2}. These results have been improved by Kukavica and O\.za\'nski~\cite{KO1,KO2} by removing stabilization terms from the model and proving an exponential decay using only structural damping and viscous fluid dissipation, treating both attached and fully immersed elastic bodies.
Other decay mechanisms in free-boundary settings have been explored in \cite{QGY}, where the authors obtained global smooth solutions and exponential energy decay for a nonlinear elastic interface model with boundary dissipation; in \cite{Co}, the author proved global existence near the equilibrium
and convergence to a flat-interface state for a Navier-Stokes fluid coupled to an undamped linear
wave equation; and in \cite{IKO}, where the authors showed that a vacuum bubble can
stabilize a three-dimensional curved-domain fluid-elastic system by controlling the spatial average
of the pressure, leading to global existence and exponential decay for small initial data. Finally, we also mention \cite{BKMT} for
treating an inviscid fluid interacting with a nonlinear two-dimensional plate, where the absence of
fluid viscosity changes the available dissipation mechanism.

In this paper, we address a three-dimensional fluid-structure interaction system consisting of the viscous incompressible Navier-Stokes equations coupled with a fourth-order damped plate equation posed on the moving upper boundary of the fluid domain. 
The coupling is imposed through the continuity of velocity at the fluid-structure interface and the balance of forces exerted by the fluid stress on the plate. 
The goal of the present paper is to prove the exponential decay of strong solutions for sufficiently small initial data in suitable Sobolev spaces. When structuring the estimates, we draw inspiration from the paper~\cite{KO1}. An important feature of our result, and a major distinction from \cite{KO1}, is that the initial plate displacement is allowed to be nonzero, which will be important in future applications. Consequently, the initial fluid domain need not coincide with the flat equilibrium domain, and the geometric coefficients arising in the Lagrangian formulation are nontrivial already at the initial time, which is a major difficulty when proving the equipartition estimate; see~\cite{IKLT1,IKLT2,KO1}.
Moreover, in contrast with global results whose smallness threshold depends on a prescribed final time, such as \cite{Le1,Le2},
our smallness condition is uniform in time. More precisely, we show that if the initial fluid velocity, plate displacement, and plate velocity are sufficiently small, then the corresponding strong solution converges exponentially to the flat equilibrium configuration throughout its interval of existence.

The proof relies on a combination of higher-order energy estimates, Stokes-type regularity bounds, and a nonlinear bootstrap argument. One of the main difficulties is to control the interaction between the nonlinear fluid dynamics and the fourth-order plate equation at the moving interface. A key step is to combine the differentiated fluid equations with suitably chosen test functions in order to derive the energy identities \eqref{EQ59} and~\eqref{EQ68}. These identities reveal a cancellation structure at the fluid-structure interface and allow us to construct an energy functional whose time evolution contains the dissipation needed to close the argument; see~\eqref{EQ69}.
An essential ingredient is the Poincar\'e inequality, which provides the coercivity needed to control the lower-order components of the energy by the dissipation and thereby obtain exponential decay.
Another important difficulty we face concerns the pressure, which is the main reason for the immersed fluid-structure system may not be stable in the non-flat configuration~\cite{IKO}. Since the pressure is determined implicitly through the incompressibility condition and the coupling with the moving boundary, it is not directly controlled by the basic energy. To overcome this issue, we derive Stokes-type regularity estimates that control spatial derivatives of the pressure in terms of the velocity and plate variables. We also obtain a lower-order pressure estimate by using the compatibility condition on the pressure, which fixes the additive time-dependent constant and yields control of the full pressure in terms of the total energy. These estimates are essential for closing the nonlinear bootstrap scheme and proving the exponential decay.

The paper is organized as follows. In Section~\ref{sec01}, we introduce the mathematical formulation of the model in Lagrangian coordinates and state the main result. In Section~\ref{sec02}, we establish preliminary estimates for the flow map and prove the Stokes-type estimates needed to control the velocity and pressure. In Section~\ref{sec03}, we derive the higher-order energy identities and the main energy inequality. In Section~\ref{sec04}, we provide the necessary pressure estimates, close the nonlinear bootstrap argument and complete the proof of the main result.

\colb
\startnewsection{The model and main result}{sec01} 
\subsection{Description of the model}
Let $\tt=(0,1)^2\subset \RR^2$ be the unit torus. We assume that, for each time $t\geq 0$, the fluid domain $\Omega_{\eta}(t)$ is a subgraph of a space-time dependent function $\etah\colon \tt \times[0,+\infty)\to {\mathbb R}$,
which represents the height of the moving upper boundary; see Figure~\ref{fig1}. 
More precisely,
\begin{equation*}
	\Omega_{\eta}(t)=
	\{  (\bx, z) : \bx=(x_1,x_2) \in \tt,
	z\in (0,\eta(\bx,t))\}\subseteq\R^3.
\end{equation*}
The fluid is assumed to be incompressible and Newtonian, and its motion is governed by the Navier-Stokes equations
\begin{align}
	&
	\partial_{t} u
	- 
	\Delta u
	+ u\cdot \nabla u
	+ 
	\nabla p
	= 0 \inin{\Omega_\eta(t)\times (0,T)},
	\label{EQ01}
	\\&
	\diver u=0 \inin{\Omega_\eta(t)\times (0,T)}
	,
	\label{EQ02}
\end{align}
where $u\colon  \Omega_{\eta} (t) \to \RR^3$ and $p\colon \Omega_{\eta} (t) \to \RR$ denote the velocity and pressure, respectively.
The height of the moving upper boundary $\eta(\bx,t)$ is described by a fourth-order square root damped plate equation
\begin{equation}
	\etah_{tt}
	- 
	\Delta_{\bx} \eta_t
	+ 
	\Delta_{\bx}^2 \etah
	= \fP 
	\inin{\tt \times (0,T)}
	,
	\label{EQ03}
\end{equation}
where $\Delta_{\bx}:=\partial_{11}+\partial_{22}$ is the horizontal Laplacian.
In \eqref{EQ03}, $\fP(\bx,t)\in\mathbb{R}$ denotes the density of the fluid force acting on the structure in the vertical direction $e_3=(0,0,1)$. More precisely, 
\begin{equation}
	%\fP=-S_{\eta}(-p I+(\nabla u+\nabla^{\tau} u))n^{\eta}\cdot e_3,
	\fP
	=
	-S_{\eta}
	(-p I+\nabla u) (\bx, \eta (\bx,t),t)
	n^{\eta}\cdot e_3
	\inin{\tt \times (0,T)},
	\label{EQ04}
\end{equation}
where $S_{\eta} (\bx,t)=\sqrt{1+|\nabla \eta (\bx,t)|^2}$ is the surface element of the deformed plate and $n^{\eta}$ is the unit outward normal on $\Gamma(t)$, which is the graph of~$\eta (\bx, t)$. Thus,
\begin{align}
	n^{\eta}
	= 
	\frac{1}{\sqrt{1+  |\nabla \eta|^2}}
	(-\partial_{1} \eta, -\partial_{2} \eta, 1).
	\llabel{EQ05}
\end{align}
At the lower and upper boundaries of the fluid domain, we impose the no-slip and the velocity-matching conditions, respectively, i.e.,
\begin{align}
	u(\bx,0,t)&=0
	\onon{\tt \times (0,T)},
	\label{EQ06}
	\\
	u(\bx, \eta(\bx,t),t)
	&=(0,0,\eta_t)
	\onon{\tt \times (0,T)}.
	\label{EQ07}
\end{align}

\begin{figure}[ht]
	\centering
	\begin{tikzpicture}[
		scale=1,
		x={(-0.75cm,-0.45cm)},
		y={(1.05cm,-0.08cm)},
		z={(0cm,1.00cm)},
		>=Latex,
		axis/.style={->,semithick},
		bdry/.style={black,semithick},
		hidden/.style={black,semithick,densely dashed},
		top/.style={thick},
		topdash/.style={thick,densely dashed}
		]
		
		% Dimensions of one periodic cell
		\def\Lone{2.7}
		\def\Ltwo{3.5}
		\def\H{1.55}
		
		% A sample initial height function eta_0(x_1,x_2)
		% This is only for drawing the wavy top boundary.
		\def\etaz#1#2{\H + 0.20*sin(140*(#1)+20) + 0.09*cos(120*(#2)+15)}
		
		% Corners of the bottom periodic cell
		\coordinate (O) at (0,0,0);
		\coordinate (A) at (\Lone,0,0);
		\coordinate (B) at (0,\Ltwo,0);
		\coordinate (C) at (\Lone,\Ltwo,0);
		
		% Coordinate axes
		\draw[axis] (0,0,0) -- (3.35,0,0) node[below left] {$x_1$};
		\draw[axis] (0,0,0) -- (0,4.15,0) node[below right] {$x_2$};
		\draw[axis] (0,0,0) -- (0,0,2.75) node[left] {$z$};
		
		% Bottom boundary z=0
		\draw[bdry] (O) -- (A);
		\draw[bdry] (O) -- (B);
		\draw[bdry] (A) -- (C);
		\draw[bdry] (B) -- (C);
		
		% Vertical dashed guide lines
		\draw[hidden] (0,0,0) -- (0,0,{\etaz{0}{0}});
		\draw[hidden] (\Lone,0,0) -- (\Lone,0,{\etaz{\Lone}{0}});
		\draw[hidden] (0,\Ltwo,0) -- (0,\Ltwo,{\etaz{0}{\Ltwo}});
		\draw[hidden] (\Lone,\Ltwo,0) -- (\Lone,\Ltwo,{\etaz{\Lone}{\Ltwo}});
		
		% Upper free boundary Gamma(eta_0): front and back curves
		\draw[topdash]
		plot[domain=0:\Lone,samples=100,variable=\s]
		({\s},{0},{\etaz{\s}{0}});
		
		\draw[topdash]
		plot[domain=0:\Lone,samples=100,variable=\s]
		({\s},{\Ltwo},{\etaz{\s}{\Ltwo}});
		
		% Upper free boundary: side curves
		\draw[topdash]
		plot[domain=0:\Ltwo,samples=100,variable=\s]
		({0},{\s},{\etaz{0}{\s}});
		
		\draw[topdash]
		plot[domain=0:\Ltwo,samples=100,variable=\s]
		({\Lone},{\s},{\etaz{\Lone}{\s}});
		
		% A few interior dashed vertical lines to suggest the domain
%		\draw[hidden] (1.35,0,0) -- (1.35,0,{\etaz{1.35}{0}});
%		\draw[hidden] (1.35,\Ltwo,0) -- (1.35,\Ltwo,{\etaz{1.35}{\Ltwo}});
		
		% Labels
		\node at (1.35,1.85,2.05) {$\Gamma(t)$};
%		\node at (1.35,1.75,0.75) {$\Omega_{\eta_0}$};
		\node at (1.25,1.55,-0.18) {$z=0$};
		
		% Periodic cell labels
%		\draw[<->,thin] (0,-0.38,0) -- (\Lone,-0.38,0)
%		node[midway,below] {$L_1$};
		
%		\draw[<->,thin] (-0.35,0,0) -- (-0.35,\Ltwo,0)
%		node[midway,below] {$L_2$};
		
%		\node[font=\scriptsize] at (1.35,-0.78,0) {$x_1\sim x_1+L_1$};
%		\node[font=\scriptsize] at (-0.75,1.75,0) {$x_2\sim x_2+L_2$};
		
	\end{tikzpicture}
	\caption{Configuration of the fluid domain $\Omega_{\eta} (t)$}
	\label{fig1}
\end{figure}
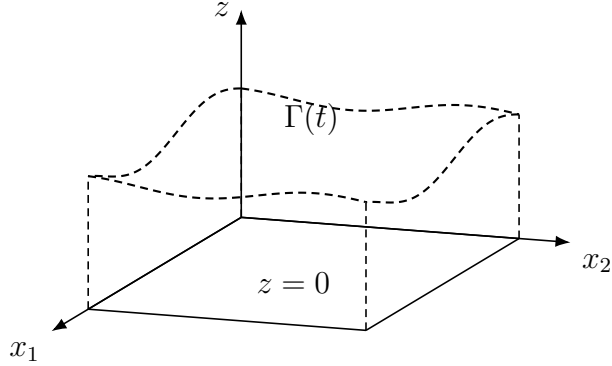
Finally, we impose the initial conditions $u(\cdot,0) = u_0$, $\eta(\cdot,0) = \eta_0$ and  $\eta_t(\cdot,0) = \eta_1$ without detailing on compatibility conditions.

\subsection{Reformulation and the main result}
Now, we rewrite the fluid-plate system in the Lagrangian coordinates.
Denote by
\begin{align}
	\eta_0(x)=1+h_0(x), \qquad x\in \mathbb T^2
	\llabel{EQ01a}
\end{align}
the initial height of the fluid,
where
\begin{align}
	\int_{\mathbb T^2} h_0(x)\,dx=0.
	\llabel{EQmean}
\end{align}
We assume that $\eta_0(x)\geq c_0$ for any $x\in \mathbb{T}^2$ for some constant $c_0>0$.
Denote by
$\Omega=\mathbb T^2\times(0,1)$ the fixed reference domain.
We define the trajectory map $\psi(\cdot,t)\colon \Omega \to \Omega_{\eta} (t)$ by
\begin{align}
	&
	\partial_t\psi(x,z,t)=u(\psi(x,z,t),t) \inin{\Omega\times (0,T)},
	\label{EQ08}
	\\&
	\psi(x,z,0)=\psi_0(x,z):=(x, z+zh_0(x)) \inin{\Omega}.
		\label{EQ09}
\end{align}
The Lagrangian velocity $v$ and pressure $q$ are then defined by composing with $\psi$, i.e.,
\begin{align}
	v(x,z,t) &=u(\psi(x,z,t),t) \inin{\Omega\times (0,T)},
	\llabel{EQ10}
	\\
	q(x,z,t)&=p(\psi(x,z,t),t) \inin{\Omega\times (0,T)}.
	\llabel{EQ11}
\end{align}
To measure the solution as a perturbation of the flat equilibrium configuration, we introduce the shifted variables
\begin{align}
	&h(x,t)=\eta(x,t)-1 \inin{\mathbb{T}^2\times (0,T)},
	\label{EQ12}
	\\&
	\Psi(x,z,t)=\psi(x,z,t)-(x,z) \inin{\Omega\times (0,T)}.
	\label{EQ13}
\end{align}
The incompressible Navier-Stokes equations \eqref{EQ01} and \eqref{EQ02} in the Lagrangian coordinates take the form
\begin{align}
	&
	J_0\partial_t v_i
	-\partial_j( J_0 a_{jl}a_{kl} \partial_k v_i)
	+
	J_0 a_{ki} \partial_k q=0 \inin{\Omega\times (0,T)},
	\qquad i=1,2,3,
	\label{EQ14}
	\\&
	a_{ki}\partial_k v_i=0 \inin{\Omega\times (0,T)},
	\label{EQ15}
\end{align}
where $J_0(x)=\det(\nabla\psi_0)=1+h_0(x)$ is the Jacobian and $a= (\nabla \psi)^{-1}$ is the inverse deformation matrix.
The plate equation \eqref{EQ03} becomes
\begin{align}
	h_{tt}-\Delta_x h_t+\Delta_x^2 h
	=
	\left(-J_0 a_{3l} a_{kl}\partial_k v_3+q\right)\big|_{z=1}
	\inin{\mathbb{T}^2 \times (0,T)}.
	\label{EQ16}
\end{align}
The boundary conditions \eqref{EQ06} and \eqref{EQ07} on the fixed lower boundary and upper moving boundary become
\begin{align}
	&v(x,0,t)=0 \onon{\mathbb{T}^2 \times (0,T)},
\label{EQ17}
	\\&
	v(x,1,t)=(0,0,h_t) \onon{\mathbb{T}^2 \times (0,T)},
	\label{EQ18}
\end{align}
and initial fluid and plate velocities satisfy $v_0 = u_0\circ\psi_0$ and $h_1 = \eta_1$.

The following theorem is the main result of the paper.
It provides a~priori estimates for the
exponential decay of strong solutions with sufficiently small initial data.

\colb
\begin{Theorem}
	\label{T01}
	Let $(v,q,a, \Psi,\etaa)$ be a smooth solution to \eqref{EQ14}--\eqref{EQ18} on some time interval $[0,T)$, where $0<T\leq +\infty$, and set
	\begin{align}
		Y(t):=
		\Vert v\Vert_{H^3}
		+
		\Vert v_t\Vert_{H^1}
		+
		\Vert v_{tt} \Vert_{L^2}
		+
		\Vert \etaa \Vert_{H^4}
		+
		\Vert \etaa_t \Vert_{H^3}
		+
		\Vert \etaa_{tt}\Vert_{H^2}
		+
		\Vert \etaa_{ttt}\Vert_{L^2}.
		\label{EQ29a}
	\end{align}
There exist constants $C\geq 1$ and $\eps >0$ such that if $Y(0) \leq \eps$,
	then
	\begin{align}
		Y(t) \leq 
		CY(0)
		 e^{-t/C}
		\comma t\in [0,T).
		\label{EQ29}
	\end{align}
\end{Theorem}
\colb
Note that the time derivatives at $t=0$ in \eqref{EQ29a} are defined iteratively by differentiating \eqref{EQ14} and \eqref{EQ16} in time and evaluating at $t=0$.

\colb
\begin{Remark}
	\label{R01}
	{\rm	The incompressibility condition \eqref{EQ02} and the boundary conditions \eqref{EQ06}--\eqref{EQ07} imply that
		\begin{align}
			\begin{split}
				&
				\frac{d}{dt}
				\int_{\tt} \etaa (x,t) dx
				=\int_{\tt} \pt \eta (x,t) dx
				=\int_{\tt} 	
				u_3 (x, \eta(x,t),t) dx
				\\&\indeq
				=
				\int_{\partial \Omega_\eta (t)} u(x,z,t) \cdot n^\eta dS
				=\int_{\Omega_{\eta} (t)}
				\diver u(x,z) dx dz=0
				\comma t\geq 0.
				\label{EQ22}
			\end{split}
		\end{align}
		We integrate \eqref{EQ16} in spatial variables over $\tt$ to obtain
		\begin{align}
			\int_{\tt} q|_{z=1} \, dx
			=
			\int_{\tt} 
			(J_0 a_{kl} a_{3l} \pa_k v_3) |_{z=1} \,dx.
			\label{EQ23}
		\end{align}
		From the Navier-Stokes equation \eqref{EQ14}, we deduce that the pressure $q$ is determined up to a time-dependent constant.
		Taking into account the relation~\eqref{EQ23},
		the pressure $q$ can then be uniquely determined. 
		Namely, we rewrite
		\begin{align}
			q(x,z,t)
			=
			\Q (x,z,t) + c(t),
			\label{EQ24}	
		\end{align}
		where we impose
		\begin{align}
			\int_{\Omega} 
			\Q (x,z,t) \,dx dz
			=0
			\comma t\geq 0.
			\label{EQ25}	
		\end{align}
		Inserting \eqref{EQ24} into \eqref{EQ23}, we infer that $c(t)$ satisfies
		\begin{align}
			\begin{split}
				c(t) =
				\int_{\tt} 
				(J_0 a_{kl} a_{3l} \pa_k v_3 
				- 
				\Q)|_{z=1} \,dx.
				\label{EQ26}	
			\end{split}
		\end{align}
	}
\end{Remark} 
\colb

\begin{Remark}  
{\rm
In this paper, we restrict ourselves to deriving a priori estimates, and thus a sufficiently regular solution is assumed to be given. More precisely, existing methods for constructing strong solutions (see~\cite{Le1, DjTa}) do not directly yield the higher-order regularity required in our derivations. Construction of such solutions would require further arguments beyond the scope of the present paper.  Accordingly, we do not state the associated higher-order compatibility conditions for the initial data; they need to be satisfied by the assumed solution.
}
\end{Remark}

\begin{Remark}
{\rm
For simplicity, all physical coefficients in the fluid and plate equations have been normalized to one. The argument extends without essential changes to arbitrary fixed positive fluid and plate densities, fluid viscosity, structural damping, and bending stiffness. These coefficients only modify the weights in the energy and dissipation functionals, and the constants in the smallness condition and the exponential-decay estimate then depend on the coefficients.
}
\end{Remark} 

Throughout this paper, $C\geq 1$ denotes a sufficiently large constant whose value may vary from line to line. 
The notation $A\les B$ means that $A\leq CB$ for some constant $C>0$.

\section{Preliminary estimates}
\label{sec02}
\subsection{Particle trajectories}
We begin by deriving some necessary estimates for the particle
trajectory map.
Using $\det (\nabla \psi) =J_0$, which follows from \eqref{EQ02}, we get
\begin{align}
	J_0
	a_{ij} 
	= 
	\frac{1}{2}
	\eps_{imn} \eps_{jkl}
	\pa_m \psi_k \pa_n \psi_l
	=	\frac{1}{2}
	\eps_{imn} \eps_{jkl}
	(\delta_{mk}
	+
	\pa_m \Psi_k) 
	(\delta_{nl}
	+
	\pa_n \Psi_l),
	\label{EQ30}
\end{align}
where $\eps_{ijk}$ denotes the permutation symbol.
As a consequence, the Piola identity
\begin{align}
	\pa_i (J_0 a_{ij}) 
	=0
	\comma j=1,2,3
	\label{EQ31}
\end{align}
holds.
From \eqref{EQ08} and \eqref{EQ13},
we deduce that
\begin{align}
	\partial_t \Psi (x,z,t)= v(x,z,t)
	\inin{\Omega\times (0,T)}.
	\label{EQ32}
\end{align}
On the upper boundary we have $z=1$ since $\psi(x,1,t)=(x,\eta(x,t))$, and thus
\begin{align}
	\Psi(x,1,t) = 
	(0,0,h(x,t))
	\inin{\tt\times (0,T)}
	.
	\label{EQ33}
\end{align}
Since $a(0)=(\nabla\psi_0)^{-1}$ and
$J_0=1+h_0$, we have
\begin{align}
J_0a(0)
=
\begin{pmatrix}
	1+h_0 & 0 & -z\partial_1h_0 \\
	0 & 1+h_0 & -z\partial_2h_0 \\
	0 & 0 & 1
\end{pmatrix}
\label{EQJ}
\end{align}
and
\begin{align}
J_0 a(0)a(0)^T
=
\begin{pmatrix}
	1+h_0+\dfrac{z^2(\partial_1 h_0)^2}{1+h_0}
	&
	\dfrac{z^2(\partial_1 h_0)(\partial_2 h_0)}{1+h_0}
	&
	-\dfrac{z\partial_1 h_0}{1+h_0}
	\\[8pt]
	\dfrac{z^2(\partial_1 h_0)(\partial_2 h_0)}{1+h_0}
	&
	1+h_0+\dfrac{z^2(\partial_2 h_0)^2}{1+h_0}
	&
	-\dfrac{z\partial_2 h_0}{1+h_0}
	\\[8pt]
	-\dfrac{z\partial_1 h_0}{1+h_0}
	&
	-\dfrac{z\partial_2 h_0}{1+h_0}
	&
	\dfrac{1}{1+h_0}
\end{pmatrix}.
\label{EQJ2}
\end{align}

Let $\gamma_0 \in (0,1]$ be a small constant, to be determined at the end of the section,
and let 
$\gamma\in (0,\gamma_0]$.
We assume that 
\begin{align}
		\Vert h_0\Vert_{H^4} \leq \gamma
		\label{EQ34a}
\end{align}
and
\begin{align}
	(\Vert v\Vert_{L^1 H^3}
	+
	\Vert v\Vert_{L^\infty H^3}
	+
	\Vert \nabla q\Vert_{L^\infty L^2}
	)
	(1+ \Vert v\Vert_{L^1 H^3})
	\leq 
	\gamma,
	\label{EQ34}
\end{align}
where the norms are taken on $\Omega \times (0,T)$.
The constant $\gamma_0>0$ shall be chosen universally small so that the absorption arguments in Lemma~\ref{Lstokes} can apply.

The following lemma establishes the necessary estimates on the inverse deformation matrix.

\colb
\begin{Lemma}
	\label{L03}
If \eqref{EQ34a} and \eqref{EQ34} hold, 
		then 
		\begin{align}
			\begin{split}
				\Vert a_t\Vert_{L^\infty H^2}
				+
				\Vert I -J_0 a\Vert_{L^\infty H^2}
				+
				\Vert I - J_0 a a^T\Vert_{L^\infty H^2}
				+
				\Vert \nabla \Psi\Vert_{L^\infty H^2}
				\leq 
				C\gamma
				\label{EQ35}
			\end{split}
		\end{align}
	and 
	\begin{align}
		\Vert \pt (a a^T)\Vert_{H^2} 
		+ 
		\Vert  a_t \Vert_{H^2}
		\leq 
		C \Vert v\Vert_{H^3}
		\comma
		t\in (0,T),
		\label{EQ36}
	\end{align}
for a constant $C>0$.
\end{Lemma}
\colb

\begin{proof}
For each $t>0$, it follows from \eqref{EQ32} and the Fundamental Theorem of Calculus that
\begin{align}
	\Psi (x,z,t) 
	= 
	(0, zh_0 (x))
	+
	 \int_0^t v(x,z,s) ds.
	 \llabel{EQ37}
\end{align}
Consequently, we infer from \eqref{EQ34a} and \eqref{EQ34} that
\begin{align}
	\Vert \nabla \Psi\Vert_{L^\infty H^2}
	\les
	\Vert \nabla h_0\Vert_{H^2}
	+
	\Vert v\Vert_{L^1 H^3}
	\les 
	\gamma.
	\label{EQ38}
\end{align}
Using \eqref{EQ30} and \eqref{EQ38}, we obtain
\begin{align}
	\begin{split}
	\Vert J_0 a_t\Vert_{L^\infty H^2}
	\les	
	\Vert \nabla v\Vert_{L^\infty H^2}
	+
	\Vert \nabla v\Vert_{L^\infty H^2}
	\Vert \nabla \Psi\Vert_{L^\infty H^2}
	\les
	\Vert v\Vert_{L^\infty H^3}
	(1+ \Vert v\Vert_{L^1 H^3}),
	\label{EQ39a}
	\end{split}
\end{align}
which leads to
\begin{align}
	\Vert a_t\Vert_{L^\infty H^2}
	\les
	\Vert J_0^{-1}\Vert_{H^2}
	\Vert J_0 a_t\Vert_{L^\infty H^2}
	\les
	\Vert v\Vert_{L^\infty H^3}
	(1+ \Vert v\Vert_{L^1 H^3}).
		\label{EQ39}
\end{align}
By another application of the Fundamental Theorem of Calculus together with \eqref{EQ38} and \eqref{EQ39a}, we have
\begin{align}
	\begin{split}
	\Vert I - J_0 a\Vert_{L^\infty H^2}
	&\les
	\Vert I-  J_0 a(0)\Vert_{H^2}
	+
	\int_0^t 
	\Vert J_0 a_t\Vert_{H^2}
	\\&
	\les
		\Vert I- J_0 a(0)\Vert_{H^2}
		+
	\int_0^t 
	(\Vert \nabla v\Vert_{H^2}
	+
	\Vert \nabla v\Vert_{H^2}
	\Vert \nabla \Psi\Vert_{H^2}
	)	
	\\&
	\les
		\Vert I- J_0 a(0)\Vert_{H^2}
		+
	\Vert v\Vert_{L^1 H^3}
	(1+
	\Vert v\Vert_{L^1 H^3}).
		\label{EQ40}
	\end{split}
\end{align}
A similar argument shows that
\begin{align}
	\Vert I-a\Vert_{L^\infty H^2}
	\les
	\Vert I- a(0)\Vert_{H^2}
	+
	\Vert v\Vert_{L^1 H^3}
	(1+ \Vert v\Vert_{L^1 H^3}).
	\label{EQ40a}
\end{align}
Combining \eqref{EQ40} with \eqref{EQ40a}, we arrive at
\begin{align}
	\begin{split}
	\Vert I- J_0 a a^T\Vert_{L^\infty H^2}
	&\les	
	\Vert I -J_0 a\Vert_{L^\infty H^2}
	+
	\Vert J_0 a\Vert_{L^\infty H^2}
	\Vert I -  a\Vert_{L^\infty H^2}
	\\&
	\les
	\Vert I - J_0 a\Vert_{L^\infty H^2}
	(1+ \Vert I- a\Vert_{L^\infty H^2})
	+
	\Vert I-a\Vert_{L^\infty H^2}
	\\&
	\les
		\Vert I- J_0 a(0)\Vert_{H^2}
		+
		\Vert I - a(0)\Vert_{H^2}
	+
	\Vert v\Vert_{L^1 H^3}
	(1 + \Vert v\Vert_{L^1 H^3}).
	\label{EQ41}
	\end{split}
\end{align}
From \eqref{EQJ}, \eqref{EQ38}, \eqref{EQ39}--\eqref{EQ40}, and \eqref{EQ41}, we complete of the proof of~\eqref{EQ35}.

From \eqref{EQ38} and \eqref{EQ39}, 
we infer that
\begin{align}
	\begin{split}
		\Vert a_t \Vert_{H^2}
		\les
		(\Vert \nabla \Psi\Vert_{L^\infty H^2}
		+1)
		\Vert v\Vert_{H^3}
		\Vert J_0^{-1}\Vert_{H^2}
		\les
		\Vert v\Vert_{H^3}
		\label{EQ42}
	\end{split}
\end{align}
for any $t\in (0,T)$.
Using \eqref{EQ40a} and \eqref{EQ42}, we get
\begin{align}
	\begin{split}
	\Vert \pt (a a^T)\Vert_{H^2}
	&
	\les
	\Vert a_t \Vert_{H^2}
	\Vert a\Vert_{H^2}	
	\les
	(1+ \Vert I -a\Vert_{H^2})
	\Vert v\Vert_{H^3}
	\les
	\Vert v\Vert_{H^3},
	\llabel{EQ43}
	\end{split}
\end{align}
for any $t\in (0,T)$, and
\eqref{EQ36} follows.
\end{proof}

\subsection{Stokes-type estimates}
Here we recall the classical Sobolev regularity of solutions with the Dirichlet boundary conditions; see~\cite{GS,S,Galdibook}.
\colb
\begin{Lemma}
	\label{Lstokes0}
	Suppose that $u$ and $p$ are solutions to the Stokes system
	\begin{align}
		\begin{split}
			&
			- \Delta u + \nabla p =f
			\inin{\Omega},
			\\&
			\diver u= g
			\inin{\Omega},
			\\&
			u = h
			\onon{ \partial \Omega}
			.
		\end{split}
		\llabel{EQ44}
	\end{align}
Then we have
	\begin{align}
		\begin{split}
		\Vert u\Vert_{H^2}
		+
		\Vert \nabla p\Vert_{L^2}
		\leq
		C
		(\Vert f\Vert_{L^2}
		+
		\Vert g\Vert_{H^1}
		+
		\Vert h \Vert_{H^{3/2} (\partial \Omega)})
		\llabel{EQ45}
		\end{split}
	\end{align}
	and
		\begin{align}
		\begin{split}
			\Vert u\Vert_{H^3}
			+
			\Vert \nabla p\Vert_{H^1}
			\leq
			C 
			(\Vert f\Vert_{H^1}
			+
			\Vert g\Vert_{H^2}
			+
			\Vert h \Vert_{H^{5/2} (\partial \Omega)}).
			\llabel{EQ46}
		\end{split}
	\end{align}
where $C>0$ is a constant.
\end{Lemma}
\colb

The following lemma provides necessary Stokes-type estimates for the solution of the system \eqref{EQ14}--\eqref{EQ18}.
\colb
\begin{Lemma}
	\label{Lstokes}
There exist constants $\gamma_0>0$ and $C>0$ such that if \eqref{EQ34a} and \eqref{EQ34} hold, 
then
\begin{align}
	\begin{split}
		\Vert v_t \Vert_{H^2}
		+
		\Vert v\Vert_{H^3}
		+
		\Vert \nabla q\Vert_{H^1}
		+
		\Vert \nabla q_t \Vert_{L^2}
		\leq
		C
		(	
		\Vert  v_{tt} \Vert_{L^2}
		+
		\Vert v_t\Vert_{L^2}
		+
		\Vert \etaa_t\Vert_{H^3}
		+
		\Vert \etaa_{tt}\Vert_{H^2})
		\label{EQ47}
	\end{split}
\end{align}
for any $t\in (0,T)$.
\end{Lemma}
\colb

\begin{proof}
	Using the Piola identity \eqref{EQ31}, we have
	\begin{align}
		\pa_k \left(
		(\delta_{ki}
		-
		J_0
		a_{ki}) \int_{\Omega} q \right)=0
			\comma i=1,2,3.
		\llabel{EQ48}
	\end{align}
Therefore, \eqref{EQ14} and \eqref{EQ15} can be rewritten as \begin{align}
	\begin{split}
	-\Delta 
	v_i +
	\pa_i q
	&=
	-J_0 \pt v_i - \pa_j ((\delta_{jk} - J_0 a_{jl} a_{kl})
	\pa_k v_i
	)
	\\&\indeq
	+ 
	\pa_k \left(
	(\delta_{ki} - J_0  a_{ki}) 
	\left(q-\int_{\Omega} q\right)
	\right)
	\comma i=1,2,3,
	\\
	\diver v&=
	(\delta_{ki} 
	-
	J_0 a_{ki}) \pa_k v_i
	.
	\label{EQ49}
	\end{split}
\end{align}
Applying Lemma~\ref{Lstokes0} to the system \eqref{EQ49} with the boundary conditions \eqref{EQ17} and \eqref{EQ18}, we obtain
\begin{align}
	\begin{split}
	\Vert v\Vert_{H^3}
	+
	\Vert \nabla q\Vert_{H^1}
	&\les
	\Vert v_t \Vert_{H^1}
	+
	\Vert (I -J_0 a) \nabla v
	\Vert_{H^2}
	+
	\Vert \etaa_t \Vert_{H^{5/2} (\Gamma_1)}
	\\&\indeq
	+
	\Vert \nabla ((I- J_0 a a^T) \nabla v)\Vert_{H^1}
	+
	\left
	\Vert \nabla ((I -J_0 a) \left(
	q-\int_{\Omega} q
	\right))
	\right\Vert_{H^1}
	\\&
	\les
	\Vert v_t \Vert_{H^1}
	+
	\Vert I -J_0 a\Vert_{H^2} 
	\Vert v\Vert_{H^3}
	+
	\Vert \etaa_t\Vert_{H^3}
	+
	\Vert I - J_0 a a^T\Vert_{H^2}
	\Vert v\Vert_{H^3}
	\\&\indeq
	+
	\Vert I - J_0 a\Vert_{H^2}
	\left\Vert q- \int_{\Omega} q\right \Vert_{H^2}
	\\&
	\les
	\Vert v_t \Vert_{H^1}
	+ \Vert \etaa_t \Vert_{H^3}
	+
	\gamma
	(\Vert v\Vert_{H^3} + \Vert \nabla q\Vert_{H^1}),
	\label{EQ50}
	\end{split}
\end{align}
where we used Lemma~\ref{L03} and the Poincar\'e inequality in the last step.
Taking $\gamma_0>0$ sufficiently small
and recalling that $\gamma \in (0,\gamma_0]$, we infer that
\begin{align}
	\begin{split}
	\Vert v\Vert_{H^3}
	+
	\Vert \nabla q\Vert_{H^1}
	\les
	\Vert v_t \Vert_{H^1}
	+
	\Vert \etaa_t \Vert_{H^3}
	\les
	\delta
	\Vert v_t\Vert_{H^2}	
	+
	C_\delta \Vert v_t\Vert_{L^2}
	+\Vert \etaa_t\Vert_{H^3},
	\label{EQ51}
	\end{split}
\end{align}
for any $\delta \in (0,1)$, where we used Sobolev interpolation in the last inequality.

To estimate $v_t$, we differentiate \eqref{EQ49} in time and obtain
\begin{align}
	\begin{split}
	-\Delta \pt
	v_i +
	\pa_i q_t
	&=
	-J_0 \pa_{tt} v_i 
	- 
	\pt  \pa_j ((\delta_{jk} - J_0 a_{jl} a_{kl})
	\pa_k v_i
	)
	\\&\indeq
	+ 
	\pt
	(
	(\delta_{ki} - J_0 a_{ki}) 
	\pa_k  q )
	\comma i=1,2,3,	
		\\
	\diver \pt v
	&= 
	\pt (
	(\delta_{ki} 
	- 
	J_0 a_{ki}) \pa_k v_i)
	,
	\label{EQ52}
	\end{split}
\end{align}
with the boundary conditions
\begin{align}
\begin{split}
		&v_t = (0,0, \etaa_{tt}) \onon{\Gamma_1},
	\\&	
	v_t= 0 \onon{\Gamma_0}.
	\label{EQ53}
\end{split}
\end{align}
Using Lemmas~\ref{L03} and~\ref{Lstokes0}, we have
\begin{align}
	\begin{split}
		&
	\Vert v_t \Vert_{H^2}
	+
	\Vert  \nabla q_t \Vert_{L^2}	
	\les 
	\Vert v_{tt}\Vert_{L^2}
	+
	\Vert \pt \nabla ((I -J_0 a a^T) \nabla v)\Vert_{L^2}
	+
	\Vert \pt ((I- J_0 a) \nabla q)\Vert_{L^2}
	\\&\indeq
	+
	\Vert \pt ((I -J_0 a) \nabla v)\Vert_{H^1}
	+
	\Vert \etaa_{tt}\Vert_{H^{3/2}}
	\\&
	\les
	\Vert v_{tt}\Vert_{L^2}
	+
	\Vert I -J_0 a a^T \Vert_{H^2}
	\Vert v_t\Vert_{H^2}
	+
	\Vert\pt (a a^T) \Vert_{H^2}
	\Vert v\Vert_{H^2}
	+
	\Vert I -J_0 a\Vert_{H^2}
	\Vert \nabla q_t\Vert_{L^2}
	\\&\indeq
	+
	\Vert a_t\Vert_{H^2} 
	\Vert \nabla q\Vert_{L^2}
	+
	\Vert a_t\Vert_{H^2}
	\Vert v\Vert_{H^2}
	+
	\Vert I -J_0 a\Vert_{H^2}
	\Vert v_t \Vert_{H^2}
	+ 
	\Vert \etaa_{tt} \Vert_{H^2}
	\\&
	\les 
	\Vert v_{tt}\Vert_{L^2}
	+
	\gamma 
	(\Vert v_t\Vert_{H^2} 
	+
	\Vert \nabla q_t\Vert_{L^2}
	+\Vert v\Vert_{H^3}
	)
	+
	\Vert \etaa_{tt}\Vert_{H^2},
	\label{EQ54}
	\end{split}
\end{align}
where we used $\Vert v\Vert_{H^2} + \Vert \nabla q\Vert_{L^2} \leq \gamma$ in the last step due to~\eqref{EQ34}.
Finally,
combining \eqref{EQ51} and \eqref{EQ54} and taking $\gamma_0, \delta>0$ sufficiently small, we conclude the proof of~\eqref{EQ47}.
\end{proof}

\section{Energy estimates}
\label{sec03}

\subsection{Energy estimates}
Denote by
\begin{align}
	\F:=\{\partial_t, \bp, \partial_{tt}, \bp \pt, \bp^2\}
	\label{EQ56}
\end{align}
the set of tangential differential operators, where $\bp=\pa_1$ or
$\pa_2$ are the horizontal derivatives. 
With $\D \in \F$,
we apply $\D$ to \eqref{EQ14} and \eqref{EQ15} to obtain
\begin{align}
	&
	J_0 \partial_{t}\D v_i
	- 
	\D
	\partial_{j}( J_0 a_{jl} a_{kl}\partial_{k}
	v_i)
	+ 
	\partial_{k} (J_0 a_{ki} \D q)
	= 
	\pa_k (J_0 a_{ki} \D q)
	-
	\D \pa_k (J_0 a_{ki} q)
	\nonumber
	\\&\qquad\qquad
	+
	J_0\pt \D v_i - \D (J_0 \pt v_i),
	\label{EQ55}
	\\&
	a_{ki} \partial_{k}\D v_i
	= 
	-\D (a_{ki} \partial_{k} v_i) +    
	a_{ki} \partial_{k}\D v_i.
	\label{EQ57}
\end{align}
Next, we apply $\D$ to  \eqref{EQ16} to deduce
\begin{equation}
	\D \etaa_{tt}
	- 	
	\Delta_{\bx} \D \etaa_t
	+ 
	\Delta_{\bx}^2 \D \etaa
	= 
	\D 
	(- J_0 a_{kl} a_{3l} \partial_k v_3 +q)|_{z=1}.
	\label{EQ58}
\end{equation}
Testing \eqref{EQ55} with $\D v_i$ and \eqref{EQ58} with $\D \etaa_t$, respectively, and 
adding the resulting identities yields
\begin{align}
	\begin{split}
		&
		\frac12
		\frac{d}{dt}
		\left(
		\inte J_0 |\D v|^2 
		+
		\int_{\tt} |\D \etaa_t|^2
		+ 
		\int_{\tt} 
		|\Delta_{\bx} \D \etaa |^2
		\right)
		+
	 \int_{\Omega} |S \nabla v|^2
		+
		\int_{\tt} 
		|\nabla_{\bx} \D \etaa_t|^2
		\\&\indeq
		=
		\inte 
	(	
	\pa_k (J_0 a_{ki} \D q)
	-\D \pa_k (J_0 a_{ki} q)	
	) \D v_i 
		+
		\inte J_0(	-\D (a_{ki} \partial_{k} v_i) 
		+    
		a_{ki} \partial_{k}\D v_i) 
		\D q
		\\&\indeq\indeq
		+
		\int_{\Omega}
		(\delta_{jk} - J_0 a_{jl} a_{kl} ) 
		\pa_k \D v_i \pa_j \D v_i
		-
		\int_{\Omega} 
		(\D (J_0 a_{jl} a_{kl} \pa_k v_i) - J_0 a_{jl} a_{kl} \pa_k \D v_i) 
		 \pa_j \D v_i
		 \\&\indeq\indeq
		 +
		 \int_{\Omega}
		( J_0 \pt \D v_i - \D (J_0 \pt v_i)) \D v_i
		,
	\end{split}
	\label{EQ59}
\end{align}
where we used \eqref{EQ57} and the boundary conditions $\D v_3 = \D \etaa_t$ and $J_0 a_{33}=1$ on $\Gamma_1 \times (0,T)$, which follow from \eqref{EQ18}, \eqref{EQ30}, and~\eqref{EQ33}.
In \eqref{EQ59}, $\nabla_x:= (\pa_1, \pa_2)$ denotes the horizontal gradient.

Testing \eqref{EQ58} with $\D \etaa$, we obtain
 \begin{align}
 	\begin{split}
 		&
 	\frac{d}{dt}
 	\int_{\tt} 
  \D \etaa_t \D \etaa
 		+
 		\frac{1}{2}
 		\frac{d}{dt}
 		\int_{\tt} 
 		|\nabla_{\bx} \D \etaa|^2
 		+
 		\int_{\tt} |\Delta_x \D \etaa|^2
 		-
 		\int_{\tt} | \D \etaa_t|^2
 		\\&\indeq
 		=
 		-
 		\int_{\tt} \D (J_0 a_{kl} a_{3l} \pa_k 
 		 v_3) |_{z=1} 
 		 \D \etaa 
 		+
 		\int_{\tt} 
 	\D q|_{z=1} 	\D \etaa 
 		.
 		\label{EQ60}
 	\end{split}
 \end{align}
To derive a corresponding identity for the fluid equation \eqref{EQ55}, we
introduce the test function
\begin{align}
	\Phi (x,z,t)
	&:=
	\D \Psi (x,z,t) 
	- 
	\D \Psi (x,z,\tau)
	+ 
	(0,0, z \D \etaa (x,\tau))
	\inin{\Omega\times (0,T)},
	\label{EQ61}
\end{align}
where $\tau \in [0,t]$ is fixed.
Note that $\D\Psi (x,z,\tau)=0$ and $\D\etaa (x,\tau)=0$ for $\D \in \{\pt, \bp\pt,\pt^2\}$.
Integrating \eqref{EQ18} in time from $\tau$ to $t$, we obtain
\begin{align}
	\Psi (x,1,t)- \Psi (x,1,\tau)
	=(0,0, \etaa(x,t) - \etaa (x,\tau)).
	\llabel{EQ62}
\end{align}
Applying $\D$ to this identity gives
\begin{align}
	\D \Psi (x,1,t)
	- \D \Psi (x,1,\tau)
	=(0,0, \D \etaa (x,t) - \D \etaa (x,\tau))
	\label{EQ63}.
\end{align}
From \eqref{EQ61} and \eqref{EQ63} it follows that
\begin{align}
	\begin{split}
		&\Phi (x,1,t)= (0, 0, \D h(x,t) )
	\onon{\tt \times (0,T)},
	\\&
	\Phi (x,0,t)= 0
	\onon{\tt \times (0,T)}
	.
	\end{split}
	\label{EQ64}
\end{align}
Testing \eqref{EQ55} with $\Phi_i$, we arrive at
\begin{align}
\begin{split}
	&\frac{d}{dt}
	\int_\Omega J_0 Sv_i\Phi_i
	-
	\int_\Omega J_0|Sv|^2
	+
	\int_\Omega S\bigl(J_0a_{jl}a_{kl}\partial_kv_i\bigr)\partial_j\Phi_i
	-
	\int_{\mathbb T^2}S
	\bigl(J_0a_{3l}a_{kl}\partial_kv_3\bigr)\Phi_3\big|_{z=1}
 	\\&\qquad
	-
	\int_\Omega
	\D 
	(J_0 a_{ki}q)
	\partial_k\Phi_i
	+\int_{\mathbb T^2} Sq
	\Phi_3\big|_{z=1}
	=
	\int_\Omega
	(J_0 \partial_tSv_i
	-S(J_0\partial_tv_i))\Phi_i
	,
	\label{EQ65}
\end{split}
\end{align}
where we used the boundary conditions \eqref{EQ64}, $J_0 a_{33}=1$ on $\Gamma_1 \times (0,T)$ and $\pt \Phi_i= S \pt \Psi_i = S v_i$.
The third term on the left-hand side may be rewritten as 
\begin{align}
	\begin{split}
	&\int_\Omega
	S(J_0a_{jl}a_{kl}\partial_kv_i)\partial_j\Phi_i
	 =
	\int_\Omega \partial_jSv_i\,\partial_j\Phi_i
	+
	\int_\Omega
	S((J_0a_{jl}a_{kl}-\delta_{jk})\partial_kv_i
	)
	\partial_j\Phi_i
	\\
	&\quad =
	\int_\Omega
	\nabla S\partial_t(\Psi(t)-\Psi(\tau))
	:
	\nabla(S\Psi(t)-S\Psi(\tau))
	+
	\int_\Omega
	\nabla Sv_3:\nabla(zSh(\tau))
	\\
	&\qquad
	+
	\int_\Omega
	S ((J_0a_{jl}a_{kl}-\delta_{jk})\partial_kv_i )
	\partial_j\Phi_i
	\\
	&\quad =
	\frac12\frac{d}{dt}
	\|\nabla S(\Psi-\Psi(\tau))\|_{L^2}^2
	+
	\int_\Omega
	\nabla Sv_3:\nabla(zSh(\tau))
	\\
	&\qquad
	+
	\int_\Omega
	S(
	(J_0a_{jl}a_{kl}-\delta_{jk})\partial_kv_i
	)
	\partial_j\Phi_i,
	\label{EQ66}
	\end{split}
\end{align}
while the 
fifth term on the left-hand side of \eqref{EQ65} may be recast as
\begin{align}
	\begin{split}
	- 
	\int_{\Omega}
	\D 
	(J_0 a_{ki}
	q ) \partial_k \Phi_i 
	=
	\int_{\Omega}
	\D ((\delta_{ki} -J_0 a_{ki})
	q ) \partial_k \Phi_i 	
	-
	\int_{\Omega}
	\D q \diver \Phi.
	\label{EQ67}
	\end{split}
\end{align}
Using \eqref{EQ66} and \eqref{EQ67} in \eqref{EQ65}, and then adding the resulting identity to \eqref{EQ60}, 
we obtain
\begin{align}
	\begin{split}
		&
	\frac{d}{dt}
	\left(		\int_{\tt} 
	\D \etaa_t \D \etaa
	+
	\int_{\Omega} 
	J_0 \D v_i 
	\Phi_i
	+
	\frac{1}{2}\int_{\tt} 
	|\nabla_{\bx} \D \etaa|^2
	+
	\frac{1}{2}
	\int_{\Omega}
	\vert \nabla \D (\Psi - \Psi (\tau))\vert^2
	\right)
		\\&\indeq
	+
	\int_{\tt} |\Delta_x \D \etaa|^2
	=
	\int_{\tt} | \D \etaa_t|^2
	+
	\int_{\Omega} J_0 |\D v|^2
	-
	\int_{\Omega} \nabla \D v_3 : \nabla 
	( 
	z \D \etaa (\tau))
		\\&\indeq
	-
	\int_{\Omega} \D 
	((J_0 a_{jl} a_{kl} - \delta_{jk})
	\partial_k v_i) \partial_j  \Phi_i 
	+ 
	\int_{\Omega}
	\D ((J_0 a_{ki} - \delta_{ki})  q) \pa_k \Phi_i
	+
	\int_{\Omega} \D q  \diver \Phi
	\\&\indeq
	+	
	\int_\Omega
	(J_0 \partial_tSv_i
	-S(J_0\partial_tv_i))\Phi_i
	,
	\label{EQ68}
	\end{split}
\end{align}
where the boundary terms cancel due to~\eqref{EQ64}.
We multiply \eqref{EQ68} by a sufficiently small constant $\lambda\in (0,1/C]$, to be determined in the next section, and add the resulting identity to \eqref{EQ59}, yielding
\begin{equation}
	\frac{d}{dt}
	\E_S(t,\tau)
	+
	\DD_S(t)
	\le \L_S(t,\tau)
	+
	\N_S(t,\tau)
	+
	\C_S(t),
	\label{EQ69}
\end{equation}
for each $S\in \mathcal F$, where
\begin{align}
\begin{split}
		\E_S(t,\tau)
	&:= \frac12
	\left(
	\Vert J_0^{1/2} \D v\Vert_{L^2}^2
	+
	\|Sh_t\|_{L^2}^2
	+
	\|\Delta_x Sh\|_{L^2}^2
	\right)
	+
	\frac{\lambda}{2}
	\|\nabla_x Sh\|_{L^2}^2
	\\&\quad
	+
	\frac{\lambda}{2}
	\|\nabla S(\Psi-\Psi(\tau))
	\|_{L^2}^2
	+
	\lambda\int_{\mathbb T^2} Sh_t\,Sh
	+\lambda\int_\Omega J_0Sv_i\Phi_i ,
	\label{EQ70}
\end{split}
\end{align}
\begin{align}
	\DD_S(t)
	&:= 
	\frac12\|S\nabla v\|_{L^2}^2
	+
	\frac12\|\nabla_x Sh_t\|_{L^2}^2
	+
	\lambda\|\Delta_x Sh\|_{L^2}^2
	,
	\llabel{EQ71}
\end{align}
\begin{equation}
	\begin{split}
		\L_S(t,\tau)
	:=
	-\lambda\int_\Omega \nabla Sv_3:\nabla(zSh(\tau))
	+\lambda\int_\Omega Sq\,\operatorname{div}\Phi
	=:\L_{S,1}+\L_{S,2},
	\label{EQ72}
	\end{split}
\end{equation}
\begin{align}
\begin{split}
		\N_S(t,\tau)
	&:= \int_\Omega
	(\delta_{jk}-J_0a_{jl}a_{kl})
	\partial_kSv_i\,\partial_jSv_i
	-\lambda\int_\Omega
	S
	((J_0a_{jl}a_{kl}-\delta_{jk})\partial_kv_i
	)
	\partial_j\Phi_i
	\\
	&\quad
	+\lambda\int_\Omega
	S ((J_0a_{ki}-\delta_{ki})q
	)
	\partial_k\Phi_i 
	+
	\int_\Omega
	 (J_0\partial_tSv_i-S(J_0 \partial_t v_i)
	 )Sv_i
	\\
	&\quad
	+\lambda
	\int_\Omega
	 ( J_0\partial_t
	Sv_i-S(J_0\partial_t v_i) )\Phi_i,
	\label{EQ73}
\end{split}
\end{align}
and
\begin{align}
	\begin{split}
	\C_S (t)
	&:= 
	\int_\Omega
	 (
	\partial_k (J_0 a_{ki}Sq)
	-
	S\partial_k(J_0 a_{ki}q)  )
	Sv_i
	+
	\int_\Omega
	(
	-S(J_0 a_{ki}
	\partial_k v_i)
	+
	J_0 a_{ki}\partial_k Sv_i
	) Sq
	\\&\indeq
	-\int_\Omega
	 (S(J_0a_{jl}a_{kl}\partial_kv_i)
	-J_0 a_{jl} a_{kl}\partial_k Sv_i
	)
	\partial_jSv_i
	=: 
	\C_{S,1}
	+
	\C_{S,2}
	+
	\C_{S,3}.
	\label{EQ74}
	\end{split}
\end{align}
In \eqref{EQ69}, we also used the Poincar\'e inequalities $\Vert \D v\Vert_{L^2} \leq C \Vert \nabla \D v\Vert_{L^2}$ and $\Vert \D \etaa_t \Vert_{L^2} \leq C \Vert \nabla \D \etaa_t\Vert_{L^2}$, which follow from \eqref{EQ17} and~\eqref{EQ22}.

\subsection{The total energy}
We define the total energy by
\begin{equation}
	X(t):=
	\sum_{\D\in F}
	\left(
	\|Sv\|_{L^2}^2
	+
	\|Sh_t\|_{L^2}^2
	+
	\|Sh\|_{H^2}^2
	\right),
	\llabel{EQ75}
\end{equation}
where the set \(F\) is introduced in~\eqref{EQ56}. The total energy
\(X(t)\) controls all the norms required in the nonlinear estimates.
More precisely, we claim that there exist sufficiently large constants
\(\bar{C},M\ge 1\) such that
\begin{equation}
	\|v\|_{H^3}^2
	+
	\|\nabla q\|_{H^1}^2
	+
	\|v_t\|_{H^2}^2
	+
	\|\nabla q_t\|_{L^2}^2
	\le 
	\bar{C} 
	X(t),
	\label{EQ76}
\end{equation}
for all \(t\ge 0\), as long as
\begin{equation}
	\mathcal{H}(t)
	:=
	\sup_{s\in[0,t]}
	\left(
	X^{1/2}(s)+\int_0^s X^{1/2}(\sigma)\,d\sigma
	\right)
	\le \frac{\gamma_0}{M},
		\label{EQ77}
\end{equation}
where \(\gamma_0>0\) is the fixed constant in Lemma~\ref{Lstokes}.
The proof of \eqref{EQ76} provided \eqref{EQ77} holds is similar to the proof of \eqref{EQ47} by the elliptic regularity and we present the details below.

Since the pair $(v(0), q (0))$ solves the Stokes system \eqref{EQ49} with the boundary conditions~\eqref{EQ17} and \eqref{EQ18}, Lemma~\ref{Lstokes0} implies that
\begin{align}
	\begin{split}
		\Vert v (0)\Vert_{H^2}
		+
		\Vert \nabla q (0)\Vert_{L^2}
		&\les
		\Vert v_{t} (0)\Vert_{L^2}
		+
		\Vert (I-  J_0 a) \nabla v (0)\Vert_{H^1}
		+
		\Vert \etaa_{t} (0)\Vert_{H^{3/2}}
		\\&\indeq
		+
		\Vert (I-J_0 a a^T) \nabla v (0)\Vert_{H^1}
		+
		\left\Vert (I-J_0 a) \left(
		q-\int_{\Omega} q
		\right) (0)
		\right\Vert_{H^1},
		\label{EQ77a}
	\end{split}
\end{align}
where we used the bound $\Vert h_0\Vert_{H^4}\les 1$.
By the Poincar\'e inequality, we have
\begin{align}
	\begin{split}
	\Vert v(0)\Vert_{H^2}
	+
	\Vert \nabla q(0)\Vert_{L^2}
	\les
	\Vert v_t (0)\Vert_{L^2}	
	+
	\Vert \etaa_t (0)\Vert_{H^{3/2}}.
	\label{EQ77b}
	\end{split}
\end{align}
where we used \eqref{EQJ}--\eqref{EQJ2} and the assumption that $\Vert h_0\Vert_{H^4}$ is sufficiently small.
A similar argument as in \eqref{EQ77a} and \eqref{EQ77b} implies that
\begin{align}
	\|v(0)\|_{H^3}
	+
	\|\nabla q(0)\|_{H^1}
	\les
	\|v_t(0)\|_{H^1}
	+
	\|h_t(0)\|_{H^{5/2}}.
		\label{EQ77c}
\end{align}

Note that the pair $(v_t (0), q_t (0))$ solves the Stokes systems \eqref{EQ52} with the boundary conditions~\eqref{EQ53}. Appealing to Lemma~\ref{Lstokes0}, we obtain
\begin{align}
\begin{split}
		\Vert v_t (0)\Vert_{H^2}
	+
	\Vert \nabla q_t (0)\Vert_{L^2}
	&\les
	\Vert v_{tt} (0)\Vert_{L^2}
	+
	\Vert \pt ((I-  J_0 a) \nabla v) (0)\Vert_{H^1}
	+
	\Vert \etaa_{tt} (0) \Vert_{H^{3/2}}
	\\&\indeq
	+
	\Vert \pt ((I-J_0 a a^T) \nabla v) (0)\Vert_{H^1}
	+
	\Vert \pt ((I- J_0 a) \nabla q) (0)\Vert_{L^2}.
	\label{EQ77d}
\end{split}
\end{align}
For the second term on the right-hand side, we use to \eqref{EQ30} to get
\begin{align}
	\begin{split}
	\Vert \pt ((I - J_0 a) \nabla v) (0)\Vert_{H^1}
	&\les	
	\Vert (I-  J_0 a) (0)\Vert_{H^2}
	\Vert \pt \nabla v (0)\Vert_{H^1}
	\\&\indeq
	+
	\Vert J_0 \pt a (0)\Vert_{H^2}
	\Vert \nabla v (0)\Vert_{H^1}.
	\label{EQ77e}
	\end{split}
\end{align}
The fourth and fifth terms on the right-hand side of \eqref{EQ77d} can be estimated in a similar fashion as~\eqref{EQ77e}.
Therefore, we appeal to \eqref{EQ77b} and arrive at
\begin{align}
\begin{split}
		\Vert v_t (0)\Vert_{H^2}
	+
	\Vert \nabla q_t (0)\Vert_{L^2}
	&\les
	\Vert v_{tt} (0)\Vert_{L^2}
	+
	\Vert v_t (0)\Vert_{L^2}
	+
	\Vert \etaa_t (0)\Vert_{H^{3/2}}
	+
	\Vert \etaa_{tt} (0)\Vert_{H^{3/2}}
	\\&
	\les
	X^{1/2} (0).
	\label{EQ78a}
\end{split}
\end{align}
From \eqref{EQ77c} and \eqref{EQ78a} it follows that
\begin{align}
	\|v(0)\|_{H^3}+\|\nabla q(0)\|_{H^1}
	\lesssim 
	X^{1/2}(0).
	\label{EQ78b}
\end{align}
Combining \eqref{EQ78a} with \eqref{EQ78b}, we infer that \eqref{EQ76} holds at time $t=0$ for some constant $\bar{C}\geq 1$.

Choosing \(M\ge 1\) sufficiently large, we deduce by \eqref{EQ76}, \eqref{EQ77}, and a continuity argument that
\begin{align}
	\left(
	\|v\|_{L^1H^3}
	+\|v\|_{L^\infty H^3}
	+\|\nabla q\|_{L^\infty L^2}
	\right)
	\left(
	1+\|v\|_{L^1H^3}
	\right)
	\leq \gamma_0,
	\llabel{EQ78c}
\end{align}
on some time interval $[0,t]$, where $t>0$.
Therefore, by taking $\Vert h_0\Vert_{H^4}\leq \gamma_0$, we infer from  Lemma~\ref{Lstokes} that \eqref{EQ47} holds.
Taking \(\bar{C}\) larger than \(2C\), where $C$ is the constant appearing on the right-hand side of \eqref{EQ47},
a standard bootstrap
argument yields that \eqref{EQ76} continues to hold as long as \eqref{EQ77} is satisfied.

From here on, the constant \(\bar{C}\) in \eqref{EQ76} is fixed for the
remainder of the paper, and the value of \(M\) in \eqref{EQ77} shall be
determined in the next section.

\section{Proof of the main result}
\label{sec04}
In this section, we prove Theorem~\ref{T01}.
We recall that \eqref{EQ76} remains valid for all $t\geq 0$, as long as the bootstrap assumption \eqref{EQ77} holds.
It is clear that 
\begin{align}
	\Vert \etaa_0 \Vert_{H^4}
	\leq 
	\mathcal{H} (0) 
	\leq \mathcal{H}(t).
	\llabel{EQ81a}
\end{align}
From \eqref{EQ76} it follows that 
\begin{align}
(	\Vert v\Vert_{L^1 H^3}
	+
	\Vert v\Vert_{L^\infty H^3}
	+
	\Vert \nabla q\Vert_{L^\infty L^2})
	(1+ \Vert v\Vert_{L^1 H^3})
	\leq
	C \mathcal{H} (t)
	+
	C \mathcal{H}^2 (t)
	\leq 
	C\mathcal{H} (t).
	\llabel{EQ81b}
\end{align}
Taking $M\geq 1$ sufficiently large so that 
\begin{align}
	C\mathcal{H}(t) 
	\leq 
	\frac{C \gamma_0}{M}
	\leq 
	\gamma_0,
	\llabel{EQ34b}
\end{align}
we see that the hypotheses \eqref{EQ34a} and \eqref{EQ34} of
Lemma~\ref{L03} are satisfied.
Therefore, we obtain
\begin{align}
	\begin{split}
		\Vert a_t\Vert_{L^\infty H^2}
		+
		\Vert I - J_0 a\Vert_{L^\infty H^2}
		+
		\Vert I - J_0 a a^T\Vert_{L^\infty H^2}
		+
		\Vert \nabla \Psi\Vert_{L^\infty H^2}
		\leq 
		C \mathcal{H}(t).
		\label{EQ81}
	\end{split}
\end{align}

\subsection{Pressure estimates}
First, we provide some necessary lower-order pressure estimates.
Note that the pressure can be uniquely determined due to Remark~\ref{R01}.
Let
\begin{align}
	\tilde{q} (x,z,t) := q(x,1,t)
	\inin{\Omega\times (0,T)}
	\llabel{EQ82}
\end{align}
and
\begin{align}
	\bar{q} (x,t) := q(x,1,t)
	\inin{\tt\times (0,T)}
   \llabel{EQ83}
\end{align}
be the restrictions of the pressure to the upper boundary $z=1$. 
Using the Poincar\'e inequality, we get
\begin{align}
	\begin{split}
		\Vert q\Vert_{L^2 (\Omega)}
		&\les 
		\Vert q- \tilde{q}\Vert_{L^2 (\Omega)}
		+
		\Vert \tilde{q}\Vert_{L^2 (\Omega)}	
		\les
		 \Vert \partial_3 (q-\tilde{q})\Vert_{L^2 (\Omega)}
		+
		\Vert \bar{q}\Vert_{L^2 (\tt)}
		\\&
		\les 
		\Vert \nabla q\Vert_{L^2 (\Omega)}
		+
		\Vert \bar{q}\Vert_{L^2 (\tt)}.
		\label{EQ84}
	\end{split}
\end{align}
The second term on the far right-hand side of \eqref{EQ84} can be estimated as
\begin{align}
	\begin{split}
		\Vert \bar{q}\Vert_{L^2 (\tt)}	
		&\leq
		\left
		\Vert 
		\bar{q} - \int_{\tt} \bar{q}
		\right\Vert_{L^2 (\tt)}
		+
		\left\Vert 
		\int_{\tt} \bar{q} 
		\right\Vert_{L^2 (\tt)}
		:= 
		J_1+J_2
		.
		\label{EQ85}
	\end{split}
\end{align}
For the term $J_1$, the Poincar\'e and trace inequalities imply
\begin{align}
	\begin{split}
		J_1	 
		\les
		 \Vert \nabla_x \bar{q}\Vert_{L^2 (\tt)}
		\les
		\Vert \nabla q\Vert_{H^1 (\Omega)}
		\les 
		X^{1/2},
		\label{EQ86}
	\end{split}
\end{align}
while for the term $J_2$, we appeal to \eqref{EQ23}, \eqref{EQ81}, and the trace inequality to obtain
\begin{align}
	\begin{split}
		J_2
		&
		\les
		\int_{\tt} 
		|J_0 a a^T \nabla v| 
		\big|_{z=1}
		\les
		\left(\int_{\tt} 
		|J_0 a a^T \nabla v|^2 \big|_{z=1}
		\right)^{1/2}
		\les 
		\Vert J_0 a a^T \nabla v\Vert_{H^1 (\Omega)}
		\les 
		X^{1/2} 
		,
		\label{EQ87}
	\end{split}
\end{align}
where we used the bound $\mathcal{H} (t)\les 1$.
From \eqref{EQ76} and \eqref{EQ84}--\eqref{EQ87}, we infer that
\begin{align}
	\Vert q\Vert_{H^2}
	\les
	X^{1/2}.
	\label{EQ88}
\end{align}

Now we estimate the term $\Vert q_t\Vert_{L^2}$.
We differentiate \eqref{EQ24}--\eqref{EQ26} in time to obtain
\begin{align}
	q_t 
	= 
	(\Q)_t + c'(t),
	\label{EQ89}
\end{align}
with
\begin{align}
	\int_{\Omega}
	(\Q)_t \,dxdz=0
	\label{EQ90}
\end{align}
and
\begin{align}
	c'(t)= \int_{\tt}
	\left(
	\pt
	(J_0 a_{kl} a_{3l} \pa_k v_3) 
	- 
	(\Q)_t
	\right)
	\big|_{z=1}
	\,dx.
	\label{EQ91}
\end{align}
Using \eqref{EQ89}, \eqref{EQ90}, and the Poincar\'e inequality, we get
\begin{align}
	\begin{split}
		\Vert (\Q)_t\Vert_{L^2 }
		\les 
		\Vert \nabla (\Q)_t \Vert_{L^2 }
		=
		\Vert \nabla q_t\Vert_{L^2 }
		\les
		X^{1/2}.	
		\label{EQ92}
	\end{split}
\end{align}
From \eqref{EQ81}, \eqref{EQ91}, and \eqref{EQ92}, it follows that
\begin{align}
	\begin{split}
		|c'(t)| 
		&
		\les
		\Vert \pt (J_0 a a^T \nabla v) |_{z=1} \Vert_{L^2 (\tt)}
		+
		\Vert (\Q)_t |_{z=1}\Vert_{L^2 (\tt)}
		\\&
		\les 
		\Vert \pt (J_0 a a^T \nabla v)\Vert_{H^1 (\Omega)}
		+
		\Vert (\Q)_t\Vert_{H^1 (\Omega)}
		\les X^{1/2}
		,
		\label{EQ93}
	\end{split}
\end{align}
where we used the trace inequality.
Finally, we appeal to \eqref{EQ76}, \eqref{EQ89}, \eqref{EQ92}, and \eqref{EQ93}, obtaining
\begin{align}
	\Vert q_t\Vert_{H^1}
	\les
	X^{1/2}.
	\label{EQ96}
\end{align}

\subsection{Proof of the exponential decay}
We sum the energy inequality \eqref{EQ69} for all $\D \in  \F$ and integrate in time from $\tau$ to $t$, yielding
\begin{align}
	\begin{split}
		\sum_{\D \in 
			\F
		}	
		\E_{\D} (t, \tau)
		+ 
		\sum_{\D\in \F}
		\int_{\tau}^{t} \DD_{\D} 
		\,ds
		\leq 
		\sum_{\D \in  \F}
		\left.
		\E_{\D} (t, \tau) 
		\right|_{t=\tau}
		+
		\sum_{\D \in  \F}	
		\int_{\tau}^{t} 
		(\L_{\D} 
		+ 
		\N_{\D} 
		+ 
		\C_{\D})
		\,ds.
		\label{EQ97}
	\end{split}
\end{align}
For the last term of $\E_{\D} (t,\tau)$ in \eqref{EQ70},
we appeal to Young's and Poincar\'e inequalities, obtaining
\begin{align}
	\begin{split}
		\lambda 
		\left|
		\int_{\Omega}
		\D v_i \Phi_i
		\right|
		&
		\leq
		\lambda 
		\left|
		\int_{\Omega}
		\D v_i	(\D \Psi_i  
		- 
		\D \Psi_i (\tau))
		\right|
		+ 
		\lambda 
		\left| 
		\int_{\Omega} \D v_3 z\D \etaa (\tau)
		\right|
		\\&
		\leq
		\lambda 
		C_{\delta}
		\Vert \D v\Vert_{L^2}^2
		+ 
		C \delta
		\lambda
		\Vert \nabla (\D \Psi - \D \Psi (\tau)) \Vert_{L^2}^2
		+
		C
		\lambda 
		\Vert \nabla \D \etaa (\tau)\Vert_{L^2}^2,
		%		+ 
		%		\lambda \left| 
		%		\int_{\Omega} \D v_3 z\D \etaa (\tau)
		%		\right|
		\label{EQ98}
	\end{split}
\end{align}
for any $\delta \in (0,1]$.
Taking $\delta >0$ sufficiently small, we have
\begin{align}
	\begin{split}
		\frac{\lambda}{2}
		\Vert \nabla \D (\Psi - \Psi (\tau))\Vert_{L^2}^2
		+ 
		\lambda 
		\int_{\Omega}	
		\D v_i \Phi_i
		&\geq
		\frac{\lambda}{4}
		\Vert \nabla \D (\Psi - \Psi (\tau))\Vert_{L^2}^2
		-C \lambda \Vert \D v\Vert_{L^2}^2
		\\&\indeq
		-C \lambda \Vert \nabla \D \etaa (\tau)\Vert_{L^2}^2.
		\label{EQ99}
	\end{split}
\end{align}
For the second to last term of $\E_{\D} (t,\tau)$ in \eqref{EQ70}, using the Young's inequality, we have
\begin{align}
	\begin{split}
		\lambda 
		\left|
		\int_{\tt} \D \etaa_t \D \etaa
		\right|	
		\leq
		\lambda \Vert \D \etaa_t\Vert_{L^2}^2
		+
		\lambda \Vert \D \etaa\Vert_{L^2}^2.
		\label{EQ100}
	\end{split}
\end{align}
From \eqref{EQ99} and \eqref{EQ100}, it follows that
\begin{align}
	\begin{split}
		\sum_{\D \in \F}
		\E_{\D} (t,\tau)	
		&
		\geq 
		\frac{1}{C}
		\sum_{\D \in \F}
		(\Vert \D v\Vert_{L^2}^2
		+
		\Vert \D \etaa_t\Vert_{L^2}^2
		+
		\Vert \D \etaa\Vert_{H^2}^2
		)
		+
		\frac{\lambda}{C}
		\sum_{S\in \{\pt,\pa_{tt}, \bp\pt \}}
		\Vert \nabla \D \Psi \Vert_{L^2}^2
		\\&\indeq
		+
		\frac{\lambda}{C}
		\sum_{\D \in \{\bp, \bp^2\}}
		\Vert \nabla \D (\Psi - \Psi (\tau))\Vert_{L^2}^2
		-
		C \lambda 
		\sum_{\D \in \F}
		\Vert  \D \etaa (\tau)\Vert_{H^2}^2,
		\label{EQ101}
	\end{split}
\end{align} where we used the $H^2$-elliptic regularity.
For the term $\E_{\D} (t,\tau) |_{t=\tau}$, similar arguments as in \eqref{EQ98} and \eqref{EQ100} reveal
\begin{align}
	\begin{split}
		\sum_{\D \in \F}
		\left.
		\E_{\D} (t,\tau) \right|_{t=\tau}
		&\leq
		C	
		\sum_{\D \in F}
		(\Vert \D v (\tau)\Vert_{L^2}^2
		+
		\Vert \D \etaa_t (\tau)\Vert_{L^2}^2
		+
		\Vert \D \etaa (\tau) \Vert_{H^2}^2
		)
		\\&\indeq
		+
	C\lambda
		\sum_{\D \in \{\pt, \pa_{tt},\bp \pt\}}
		\Vert \nabla \D \Psi (\tau) \Vert_{L^2}^2.
		\label{EQ102}	
	\end{split}
\end{align}
From \eqref{EQ76}, \eqref{EQ97}, \eqref{EQ101}, and \eqref{EQ102}, we deduce that
\begin{align}
	\begin{split}
		&
		X(t)
		+
		\lambda \Vert \nabla  \bp( \Psi - \Psi (\tau))
		\Vert_{L^2}^2
		+
		\lambda \Vert \nabla  \bp^2 ( \Psi - \Psi (\tau))
		\Vert_{L^2}^2
		\\&\indeq
		+
		\sum_{\D \in \F}
		\int_{\tau}^{t}
		(	\Vert \D  v\Vert_{H^1}^2
		+
		\Vert  \D \etaa_t\Vert_{H^1}^2
		+
		\lambda
		\Vert \D \etaa\Vert_{H^2}^2)
		\,ds
		\\&
		\leq
		C X(\tau)
		+  
		C \lambda
		\bigl(
		(
		\Vert \nabla v (\tau) \Vert_{L^2}^2
		+
		\Vert \nabla \bp v(\tau) \Vert_{L^2}^2
		+
		\Vert \nabla v_t (\tau) \Vert_{L^2}^2
		)
		\\&\indeq\indeq\indeq\indeq\indeq\indeq\indeq\indeq
		-
		(
		\Vert \nabla v (t) \Vert_{L^2}^2
		+
		\Vert \nabla \bp v(t) \Vert_{L^2}^2
		+
		\Vert \nabla  v_t (t) \Vert_{L^2}^2
		)
		\bigr)
		\\&\indeq
		+
		C\sum_{\D \in \F}
		\int_{\tau}^{t}
		(
		\L_{\D} + \N_{\D} + \C_{\D}
		)\,ds.
		\label{EQ103}
	\end{split}
\end{align}
Using \eqref{EQ76} and the Fundamental Theorem of Calculus, the second term on the right-hand side of \eqref{EQ103} is bounded by
\begin{align}
	\begin{split}
		&
		C\lambda
		\int_{\tau}^{t}
		\left(
		\Vert \nabla v\Vert_{L^2}
		\Vert \nabla v_t\Vert_{L^2}
		+	
		\Vert \nabla \bp v\Vert_{L^2}
		\Vert \nabla \bp v_t\Vert_{L^2}
		+	
		\Vert \nabla  v_t \Vert_{L^2}
		\Vert \nabla  v_{tt} \Vert_{L^2}
		\right)
		\,ds
		\\&\indeq
		\leq 
		C \delta
		\int_{\tau}^{t}
		(	\Vert \nabla v_t\Vert_{L^2}^2
		+
		\Vert \nabla \bp v_t\Vert_{L^2}^2
		+
		\Vert \nabla v_{tt}\Vert_{L^2}^2
		)
		\,ds
		+
		C_\delta
		\lambda^2 
		\int_{\tau}^{t}
		X(s) \,ds
		\\&\indeq
		\leq
		C\delta 
		\sum_{\D \in \F}
		\int_{\tau}^{t} \DD_{\D}
		\,ds
		+
		C_\delta
		\lambda^2 
		\int_{\tau}^{t}
		X(s) \,ds,
		\label{EQ104}
	\end{split}
\end{align}
for any $\delta \in (0,1]$, where we used Young's inequality in the last step.
Note that \eqref{EQ76} is valid for $t\geq 0$ as long as the bootstrap assumption \eqref{EQ77} holds.
In order to estimate the last term on the right-hand side of \eqref{EQ103}, we use the following lemma, the proof of which is given below.

\colb
\begin{Lemma}
	\label{L04}
	For any $\delta\in (0,1)$ and $\tau\in [0,t]$, we have
	\begin{align}
		\int_\tau^t 
		(\L_{\bp^2}
		+
		\L_{\bp})
		\,ds
		&\les
		\delta 
		\sum_{\D \in  \F}
		\int_\tau^t \DD_{\D}\,ds
		+
		C_\delta \lambda^2 (t-\tau) X(\tau)
		+
		C_\delta \mathcal{H} O(X),
		\label{EQ105}
	\end{align}
with
	\begin{align}
		\int_\tau^t 
		(\N_{\pa_{tt}} + \L_{\pa_{tt}}) 
		\,ds
		&\les
		(\delta + \mathcal{H}) 	\sum_{\D \in  \F}
		\int_\tau^t \DD_{\D} \,ds
		%		+
		%		C_\delta \lambda^2 (t-\tau) X(\tau)
		+
		C_\delta
		\mathcal H O(X)
		\label{EQ106}
	\end{align}
	and
	\begin{align}
		\int_\tau^t 
		\left(
		\L_{\pt}
		+
		\L_{\pt\bp}
		+
		\sum_{\D \in \{\bp,\pt,\pt\bp,\bp^2\}}
		\N_{\D}
		+
		\sum_{\D\in F}
		\C_{\D}
		\right) 
		\,ds
		\les
			\delta \int_{\tau}^{t}
		\DD_{\pa_{tt}}
		\,ds
		+
		C_\delta
		\mathcal{H} O(X)
		\label{EQ107}
	\end{align}
	where $O(X)$ denotes any term involving powers of $\lambda$, $C$, and $(t-\tau)$, and at least two factors chosen from $\{X^{1/2}(\tau), X^{1/2}(t), \int_\tau^t X^{1/2}(s)\,ds\}$ or one factor of the form $\int_{\tau}^{t} X(s)\,ds$.
\end{Lemma}
\colb
Using \eqref{EQ103}--\eqref{EQ107} with choosing $\delta>0$ sufficiently small and $\M\geq 1$ sufficiently large, the dissipation terms on the right-hand sides of \eqref{EQ104}--\eqref{EQ107}  
are absorbed into the left-hand side of~\eqref{EQ103}.
Thus we obtain
\begin{align}
	\begin{split}
		&
		X(t)
		+
		\lambda \int_{\tau}^{t}
		X(s) ds
		\leq
		C(1+ \lambda^2 (t-\tau)) X(\tau)
		+
		C \lambda^2 \int_{\tau}^{t} X(s) \,ds
		+
		\mathcal{H} 
		O(X),
		\label{EQ108}
	\end{split}
\end{align}
where we used
\begin{align}
	\lambda
	\int_{\tau}^{t} X (s)\,ds
	\leq
	\frac{1}{C}
	\sum_{\D \in  \F} 
	\int_{\tau}^{t}
	\DD_{\D} \,ds.
	\llabel{EQ109}
\end{align}
We now fix $M\geq 1$ and
\begin{align}
	\lambda
	=
	\frac{1}{500C^2},
	\llabel{EQ110}
\end{align}
where $C\geq 1$ is the constant in~\eqref{EQ108}. 
Note that \eqref{EQ108} is valid for $t\geq 0$ provided \eqref{EQ77} holds. 
Using Lemma~\ref{LODE}, there exists $\eps>0$ such that if $X(0)\leq \eps$, then both \eqref{EQ77} and
\begin{align}
	X(t) 
	\leq 
	30C 
	X(0)
	e^{-t/1000C^3}
	\llabel{EQ111}
\end{align}
hold for all $t\geq 0$. 
From \eqref{EQ76}, \eqref{EQ88}, and \eqref{EQ96}, we deduce that 
$Y(t)$ is equivalent to $X(t)$, up to a universal constant uniformly for any $t\in [0,T)$. 
The proof of the exponential decay estimate \eqref{EQ29} is thus completed.

It remains to prove Lemma~\ref{L04} under the bootstrap assumption~\eqref{EQ77}.

\subsection{Proof of \eqref{EQ105}}
Using the Young's inequality, we obtain (recall the notation in \eqref{EQ72})
\begin{align}
	\begin{split}
	\L_{\bp^2,1}
	+
	\L_{\bp,1}
	&\les\lambda \Vert \bp^2 v\Vert_{H^1}
	\Vert \etaa (\tau)\Vert_{H^3}
	+
	\lambda 
	\Vert \bp v\Vert_{H^1}
	\Vert \etaa (\tau)\Vert_{H^2}
	\\&
	\les
	\delta 
	(\Vert \bp^2 v\Vert_{H^1}^2
	+
	\Vert \bp v\Vert_{H^1}^2)
	+
	C_\delta	
	\lambda^2 \Vert 
	\etaa (\tau)
	\Vert_{H^3}^2,
	\llabel{EQ112}
	\end{split}
\end{align}
for any $\delta \in (0,1]$.
We integrate in time from $\tau$ to $t$ to get
\begin{align}
	\begin{split}
	\int_{\tau}^{t}
	(\L_{\bp^2,1}
	+
	\L_{\bp,1})
	\,
	ds
	\les
	\delta 
\sum_{\D \in \F}
\int_{\tau}^{t}
\DD_{\D}  \, ds
	+
	C_\delta (t-\tau)
	\lambda^2 
	 X(\tau).
	\label{EQ113}
	\end{split}
\end{align}
For $\D \in \{\bp^2, \bp\}$, the Fundamental Theorem of Calculus and \eqref{EQ15} imply
\begin{align}
	\begin{split}
	\diver  \D	(\Psi - \Psi (\tau))
	=
	\int_{\tau}^t \D ( \delta_{ki} \pa_k v_i)
	\,ds
	=
	\int_{\tau}^t \D ( (\delta_{ki} - J_0 a_{ki}) \pa_k v_i)
	\,ds,
	\llabel{EQ114}
	\end{split}
\end{align}
leading to
\begin{align}
	\begin{split}
	\L_{\bp^2,2}
	+
	\L_{\bp,2}
	&
	\les
	\lambda	
	\Vert \nabla q\Vert_{H^1} 
(	\Vert \diver \bp^2 (\Psi - \Psi (\tau))\Vert_{L^2}
	+
	\Vert \diver \bp (\Psi - \Psi (\tau))
	\Vert_{L^2}
	+
	\Vert \etaa (\tau)\Vert_{H^2}
	)
	\\&
	\les
	\lambda 
	\Vert \nabla q\Vert_{H^1}
	\left(
	X^{1/2}(\tau)
	+
	\int_{\tau}^{t}  \Vert 
	(I - J_0 a) \nabla v
	\Vert_{H^2}
	\,ds
	\right)
	\\&
	\les
	\lambda
		\Vert \nabla q\Vert_{H^1}
	\left(
	X^{1/2} (\tau) 
	+ 
	\mathcal{H} \int_{\tau}^{t} X^{1/2}\,ds
	\right).
	\label{EQ115}
	\end{split}
\end{align}
A similar proof as in \eqref{EQ50} by using \eqref{EQ81} gives
\begin{align}
\begin{split}
		\Vert v\Vert_{H^3}
	+
	\Vert \nabla q\Vert_{H^1}
	&\les
	\Vert v_t\Vert_{H^1}
	+
	\Vert \bp^2 v\Vert_{H^1}
	+
	 \mathcal{H}
	(\Vert v\Vert_{H^3} + \Vert \nabla q\Vert_{H^1})
	\les
	 \sum_{\D \in \F} \DD_{\D}^{1/2} 
	+
	 \mathcal{H} 
	X^{1/2},
	\label{EQ116}
\end{split}
\end{align}
where we used the trace inequality
\begin{align}
	\Vert v\Vert_{H^{5/2} (\Gamma_1)}
	\les 
	\Vert \bp^2 v\Vert_{H^{1/2} (\Gamma_1)}
	\les
	\Vert \bp^2 v\Vert_{H^1 (\Omega)}
	\llabel{EQ117}
\end{align}
due to \eqref{EQ18} and~\eqref{EQ22}.
Using \eqref{EQ115} and \eqref{EQ116}, we arrive at
\begin{align}
\begin{split}
	&\int_{\tau}^{t} 
	(\L_{\bp^2,2}
	+ 
	\L_{\bp,2})
	\,
	ds
	\les
	\lambda \int_{\tau}^{t} 
		\Vert \nabla q\Vert_{H^1}
		\left(
		X^{1/2} (\tau)
		+ 
		\mathcal{H}  \int_{\tau}^{s} X^{1/2}
		\right)
	\,
		ds
	\\&\indeq
	\les
	\delta 
	\sum_{\D \in F}
	\int_{\tau}^{t}
	\DD_{\D} \, ds
	+
	C_\delta \lambda^2 
	\int_{\tau}^{t}
	\left(	
	X(\tau)
	+
	\mathcal{H}^2 
	\left(	\int_{\tau}^{s}
	X^{1/2}
	\right)^2
	\right)
	\,
	ds
	\\&\indeq\indeq
	+ 
	\mathcal{H} 
	\left(
	X^{1/2} (\tau)
	\int_{\tau}^{t}  X^{1/2}
	\,
	ds
	+  
	\mathcal{H}
	\left(
	\int_{\tau}^{t} X^{1/2} 
	ds
	\right)^2
	\right)
	\\&\indeq
	\les
	\delta 
	\sum_{\D \in \F}
	\int_{\tau}^{t}
	\DD_{\D}  \, ds
	+
	C_\delta \lambda^2 (t-\tau) X(\tau)
	+
	C_\delta \mathcal{H} O(X),
	\label{EQ118}
\end{split}
\end{align}
where $\delta \in (0,1]$,
and \eqref{EQ105} follows by combining~\eqref{EQ113}
and~\eqref{EQ118}.

\subsection{Proof of \eqref{EQ106}}
For $\N_{\pa_{tt}}$, we rewrite \eqref{EQ73} as
\begin{align}
	\begin{split}
	\N_{\pa_{tt}} 
	&=	
	\lambda 
	\int_{\Omega}
	\pa_{tt} ((J_0 a_{ki} - \delta_{ki}) q) \pa_k \pt v_i 
	+
	\lambda 
	\int_{\Omega} 
	\pa_{tt} 
	(( \delta_{jk} - J_0 a_{jl} a_{kl} ) \pa_k v_i)
	\pa_j \pt v_i
	\\&\indeq
	+
	\int_{\Omega} 
	(\delta_{jk} - J_0 a_{jl} a_{kl})
	\pa_k \pa_{tt} v_i \pa_j \pa_{tt} v_i
	=:I_1+ I_2+I_3.
	\label{EQ119}
	\end{split}
\end{align}
For the first two terms, we integrate by parts in time to obtain
\begin{align}
\begin{split}
	&
	\int_{\tau}^{t}
	(I_1
	+
	I_2)
	\,ds
	\les
	\lambda 
	\int_{\tau}^{t} 
	\Vert \nabla v_{tt}\Vert_{L^2}
	(	\Vert \pt ((I - J_0 a) q)\Vert_{L^2}
	+
	\Vert \pt ((I - J_0 a a^T) 
	\nabla v)\Vert_{L^2}
	)\, ds
	\\&\indeq\indeq
	+
	\lambda
	\left.
	\left(
	\int_{\Omega}
	\pt ((J_0 a_{ki} - \delta_{ki}) q)
	\pa_k \pt v_i
	+
	\int_{\Omega} 
	\pt ((\delta_{jk}
	- 
	J_0 a_{jl} a_{kl}) \pa_k v_i)
	\pa_j \pt v_i
	\right)
	\right|^{t}_{\tau}.
	\label{EQ120}
\end{split}
\end{align}
From \eqref{EQ76}, \eqref{EQ81}, \eqref{EQ88}, and \eqref{EQ96}, we infer that
\begin{align}
	\begin{split}
	\Vert \pt 
	((I - J_0 a) q)\Vert_{H^1}
	\les
	 \Vert a_t \Vert_{H^2}
	\Vert  q\Vert_{H^2}
	+
	 \Vert I - J_0 a\Vert_{H^2}
	\Vert q_t \Vert_{H^1}
	\les 
	\mathcal{H} X^{1/2}
	\label{EQ121}
	\end{split}
\end{align}
and
\begin{align}
	\begin{split}
	\Vert \pt 
	((I- J_0 a a^T) \nabla v) \Vert_{L^2}
	\les
	\Vert (a a^T)_t \Vert_{H^2}
	\Vert v\Vert_{H^2}
	+
	\Vert I - J_0 a a^T\Vert_{H^2}
	\Vert \nabla v_t\Vert_{L^2}
	\les	
	\mathcal{H} X^{1/2},
		\label{EQ122}
	\end{split}
\end{align}
since $\mathcal{H} \les 1$.
The last term on the right-hand side of \eqref{EQ120} is bounded by
\begin{align}
	\begin{split}
	&\Vert \pt ((J_0 a-I) q) (t)\Vert_{L^2}
	\Vert v_t (t)\Vert_{H^1}
	+
		\Vert \pt ((J_0 a-I) q) (\tau)\Vert_{L^2}
	\Vert v_t (\tau)\Vert_{H^1}
	\\&\indeq
	+
	\Vert \pt ((I- J_0a a^T) \nabla v) (t)\Vert_{L^2}
	\Vert v_t (t)\Vert_{H^1}
		+
	\Vert \pt ((I- J_0 a a^T) \nabla v) (\tau)\Vert_{L^2}
	\Vert v_t (\tau)\Vert_{H^1}
	\\&
	\les 
	\mathcal{H} O(X),
	\label{EQ123}
	\end{split}
\end{align}
where we appealed to \eqref{EQ121} and~\eqref{EQ122}.
Combining \eqref{EQ120}--\eqref{EQ123} and the Young's inequality, we obtain
\begin{align}
	\begin{split}
		\int_{\tau}^{t}
	(I_1+I_2)
	\,ds
	\les
	\delta
	\int_{\tau}^t  \DD_{\pa_{tt}}
	\,ds
	+
	C_\delta \mathcal{H} O(X).
	\label{EQ124}	
	\end{split}
\end{align}
For the term $I_3$, we use \eqref{EQ81} to get
\begin{align}
	\begin{split}
		\int_{\tau}^{t}
	I_3 
	\,ds
	\les	
	\int_{\tau}^t
	\Vert I -J_0 a a^T\Vert_{H^2}
	\Vert \nabla \pa_{tt} v\Vert_{L^2}^2
	\,ds
	\les
	\mathcal{H}
	\int_{\tau}^{t} \DD_{\pa_{tt}} 
	\,ds.
	\label{EQ125}
	\end{split}
\end{align}

For $\L_{\partial_{tt}}$, we first note that
\begin{align}
	\begin{split}
	\Vert \pt ((I- J_0 a) \nabla v)\Vert_{H^1}	
	\les
	 \Vert a_t\Vert_{H^2}
	\Vert v\Vert_{H^3}
	+
	 \Vert I- J_0 a\Vert_{H^2}
	\Vert v_t\Vert_{H^2}
	\les
	 \mathcal{H} X^{1/2}
	\label{EQ126}
	\end{split}
\end{align}
and
\begin{align}
	\begin{split}
		\Vert \pa_{tt} ((I- J_0 a) \nabla v)\Vert_{L^2}
		&	
		\les
		\Vert I- J_0 a\Vert_{H^2}
		\Vert v_{tt}\Vert_{H^1}
		+
		\Vert a_t\Vert_{H^2}
		\Vert v_t\Vert_{H^1}
		+
		 \Vert a_{tt}\Vert_{L^2} \Vert v\Vert_{H^3}
		\\&
		\les
		\mathcal{H} \DD_{\pa_{tt}}^{1/2}
		+
		 \mathcal{H} X^{1/2},
		\label{EQ127}
	\end{split}
\end{align}
where we used \eqref{EQ81} and
\begin{align}
	\begin{split}
		\Vert a_{tt} \Vert_{H^1}
		\les	
		\Vert v_t\Vert_{H^2}
		\Vert \nabla \Psi\Vert_{H^2}
		+
		\Vert v_t\Vert_{H^2}
		+
		\Vert v\Vert_{H^3}^2
		\label{EQ128}
	\end{split}
\end{align}
due to~\eqref{EQ30}.
We integrate by parts in time and use \eqref{EQ15}, obtaining
\begin{align}
	\begin{split}
		&	\int_{\tau}^{t}
		\L_{\pa_{tt}}	
		\,ds
		=
		\lambda \int_{\tau}^{t}
		\int_{\Omega}
		\pa_{tt} q \diver \pa_{tt}\Psi
		\,ds
		=
		\lambda \int_{\tau}^{t}
		\int_{\Omega}
		\pa_{tt} q 
		\pt
		(\delta_{ki}  \pa_k v_i)
		\,ds
		\\&\indeq
		=
		\lambda \int_{\tau}^{t}
		\int_{\Omega}
		\pa_{tt} q 
		\pt
		((\delta_{ki} - J_0 a_{ki})  \pa_k v_i)
		\,ds
		\\&\indeq
		=
		\lambda 
		\int_{\Omega}
		\pa_{t} q 
		\pa_{t}
		((\delta_{ki} - J_0 a_{ki})  \pa_k v_i)
		\Big|_{\tau}^t
		-
		\lambda \int_{\tau}^{t}
		\int_{\Omega}
		\pa_{t} q 
		\pa_{tt}
		((\delta_{ki} - J_0 a_{ki})  \pa_k v_i)
		\,ds.
		\llabel{EQ129a} 
	\end{split}
\end{align}
Therefore, using \eqref{EQ96}, \eqref{EQ126}, and \eqref{EQ127}, we get
\begin{align}
	\begin{split}
	&	\int_{\tau}^{t}
	\L_{\pa_{tt}}	
	\,ds
	\les
	\Vert q_t (t)\Vert_{L^2}
	\Vert \pt ((I-J_0 a) \nabla v) (t)\Vert_{L^2}
	+
	\Vert q_t (\tau)\Vert_{L^2}
	\Vert \pt ((I-J_0 a) \nabla v) (\tau)\Vert_{L^2}
	\\&\indeq\indeq
	+
	\int_{\tau}^{t}
	\Vert q_t\Vert_{L^2}
	\Vert \pa_{tt} ((I- J_0 a) \nabla v)\Vert_{L^2}
%	\\&\indeq
	\les
	\delta 
	\sum_{\D \in  \F}
	\int_{\tau}^{t}
	\DD_{\D} \,ds
	+
	C_\delta
	\mathcal{H} 
	O(X).
	\label{EQ129} 
	\end{split}
\end{align}
Combining \eqref{EQ119}, \eqref{EQ124}, \eqref{EQ125}, and \eqref{EQ129}, we conclude the proof of~\eqref{EQ106}.

\subsection{Proof of \eqref{EQ107}}
For the term $\L_{\pt\bp}$, we use~\eqref{EQ15}, \eqref{EQ76}, \eqref{EQ81}, and~\eqref{EQ96} to obtain
\begin{align}
	\begin{split}
		\int_{\tau}^{t}
	\L_{\pt\bp}
	\,ds
	&
	=
	\lambda 
	\int_{\tau}^{t}
	\int_{\Omega}
	 \bp q_t 
	\bp \diver v
	\,ds
	=
	\lambda \int_{\tau}^{t
	} \int_{\Omega}
	\bp q_t 
	\bp ((\delta_{ki} - J_0 a_{ki}) \pa_k v_i)
	\,ds
	\\&
	\les
\int_{\tau}^{t}
	\Vert q_t\Vert_{H^1}
	\Vert I - J_0 a\Vert_{H^2} 
	\Vert v\Vert_{H^2}
	\les 
	\mathcal{H} O(X),
	\llabel{EQ130}
	\end{split}
\end{align}
and a similar argument provides the same bound for $\int_{\tau}^{t} \L_{\pt} \,ds$.
For the term $\N_{\pt \bp}$, from \eqref{EQ76} and \eqref{EQ81} it follows that
\begin{align}
	\begin{split}
	\int_{\tau}^{t}
	\N_{\pt \bp}
	\,ds
	&=
			\int_{\tau}^{t}
	\int_{\Omega}
	(\delta_{jk} - J_0 a_{jl} a_{kl} ) 
	\pa_k \pt \bp v_i 
	\pa_j \pt \bp v_i
	\,ds
		-
		\lambda
	\int_{\tau}^{t}
	\int_{\Omega} \bp J_0 \pa_{tt} v_i 
	\bp v_i
	\\&\indeq
	-
	\lambda
			\int_{\tau}^{t}
	\int_{\Omega} 
	\pt \bp
	(( J_0 a_{jl} a_{kl} - \delta_{jk})
	\partial_k v_i) \partial_j 
	\bp v_i 
	\,ds
		-
	\int_{\tau}^{t}
	\int_{\Omega} \bp J_0 \pa_{tt} v_i \bp \pt v_i
	\\&\indeq
	+
	\lambda
			\int_{\tau}^{t}
	\int_{\Omega} \pt \bp 
	((J_0 a_{ki} - \delta_{ki})
	q) \partial_k  
	\bp v_i 
	\,ds
	\\&
	\les
			\int_{\tau}^{t}
	\Vert I -J_0 a a^T\Vert_{H^2}
	\Vert v_t\Vert_{H^2}^2
	\,ds
	+
			\int_{\tau}^{t}
	\Vert v\Vert_{H^2}
	\Vert \pt ((I - J_0 a a^T) \nabla v)\Vert_{H^1}
	\,ds
	\\&\indeq
	+
			\int_{\tau}^{t}
	\Vert \pt ((I -J_0 a) q)\Vert_{H^1}
	\Vert v\Vert_{H^2}
	\,ds
	+
	\mathcal{H} \int_{\tau}^{t}
	(\DD_{\partial_{tt}}
	+
	\DD_{\bp \pt}
	+
	\DD_{\bp}
	)\,ds.
%	\\&\indeq
%	+
%	\mathcal{H} \int_{\tau}^{t}
%	\Vert \pa_{tt} v \Vert_{L^2} \Vert\Phi \Vert_{L^2}\,ds.
		\label{EQ131a}
	\end{split}
\end{align}
Using \eqref{EQ121} and the Poincar\'e inequality, we obtain
\begin{align}
	\begin{split}
	\int_{\tau}^{t}
	\N_{\pt \bp} \,ds
		\les 
	\mathcal{H} O(X)
	+
	\mathcal{H} \sum_{S\in F} 
	\int_{\tau}^{t} 
	\DD_{\D}\,ds,
		\label{EQ131}	
	\end{split}
\end{align}
where we also used the estimate
\begin{align}
	\Vert \pt ((I- J_0 a a^T) \nabla v)\Vert_{H^1}
	\les 
	\mathcal{H} X^{1/2},
	\llabel{EQ132}
\end{align}
which follows from a similar argument as in~\eqref{EQ126}.
The term $\int_{\tau}^{t} \N_{\pt}ds$ can be estimated in a similar fashion as in~\eqref{EQ131a}--\eqref{EQ131}.

Next, we estimate the terms $\int_{\tau}^{t} \N_{\D}\,ds$, where $\D \in \{\bp, \bp^2\}$.
Note that
\begin{align}
\begin{split}
	\Vert \Phi\Vert_{H^1}
	&\les
	\Vert \D \Psi (t) - \D \Psi (\tau)\Vert_{H^1}
	+
	\Vert \D \etaa (\tau)\Vert_{H^1}
	\les
	\int_{\tau}^{t}
	\Vert v\Vert_{H^3} \, ds
	+
	\Vert \etaa (\tau) \Vert_{H^3},
	\label{EQ134}
\end{split}
\end{align}
where we used the Fundamental Theorem of Calculus in the last step.
Denote 
\begin{align}
	\N_{\D}:= 
	\N_{\D,1} +\N_{\D,2},
	\llabel{EQ133a}
\end{align}
where $\N_{\D,1}$ is the sum of the first three terms and $\N_{\D,2}$ is the sum of the last two terms in~\eqref{EQ73}. 
Using \eqref{EQ81}, \eqref{EQ88}, and \eqref{EQ134}, we obtain
\begin{align}
	\begin{split}
		\int_{\tau}^{t}
	\N_{\D,1}	
	\,ds
	&
	=
			\int_{\tau}^{t}
	\int_{\Omega}
	(\delta_{jk} - J_0 a_{jl} a_{kl} ) 
	\pa_k \D v_i \pa_j \D v_i
	\,ds
	-
	\lambda
			\int_{\tau}^{t}
	\int_{\Omega} \D 
	((J_0 a_{jl} a_{kl} - \delta_{jk})
	\partial_k v_i) \partial_j  \Phi_i 
	\,ds
	\\&\indeq
	+
	\lambda
			\int_{\tau}^{t}
	\int_{\Omega} \D 
	((J_0 a_{ki} - \delta_{ki})
	q) \partial_k  \Phi_i 
	\,ds
	\\&
	\les
	\int_{\tau}^{t}
	\Vert I - J_0 a a^T\Vert_{H^2}
	\Vert v\Vert_{H^3}^2
	\,ds
	+
	\int_{\tau}^{t}
	\Vert I- J_0 a a^T \Vert_{H^2} 
	\Vert v\Vert_{H^3}
	\Vert \Phi\Vert_{H^1}
	\,ds
	\\&\indeq
	+
	\int_{\tau}^{t}
	\Vert I - J_0 a \Vert_{H^2}
	\Vert
	q\Vert_{H^2}
	\Vert \Phi\Vert_{H^1}
	\,ds
	\\&
	\les
	\mathcal{H} 
	O(X),
	\llabel{EQ133}
	\end{split}
\end{align}
for each $\D \in \{\bp, \bp^2\}$.
For the term $\N_{\bp^2,2}$, we have
\begin{align}
	\begin{split}
	\int_{\tau}^{t}
	\N_{\bp^2,2}
	\,ds	
	&=	
	-2\int_{\tau}^{t}
	\int_{\Omega}
	\bp J_0 \pt \bp v_i \bp^2 v_i
	-
	\int_{\tau}^{t}
	\int_{\Omega}
	\bp^2 J_0 \pt v_i \bp^2 v_i
		\\&\indeq
		-2
	\lambda
	\int_{\tau}^{t}
	\int_{\Omega}
	\bp J_0 \pt \bp v_i \Phi_i
	-
		\lambda
	\int_{\tau}^{t}
	\int_{\Omega}
	\bp^2 J_0 \pt v_i \Phi_i
	\\&
	\les
	\mathcal{H}
	\int_{\tau}^{t} (\DD_{\bp \pt} + \DD_{\bp^2}
	+\DD_{\pt}
	)\,ds
	+
	\mathcal{H} O(X)
	.
	\label{EQ134b}
	\end{split}
\end{align}
The term $\N_{\bp,2}$ can be estimated in a similar fashion as in~\eqref{EQ134b}.

Finally, it remains to estimate the commutator terms $\int_{\tau}^{t} \C_{\D}\,ds$ for $\D\in \F$.
From \eqref{EQ74}, we note that for each term in $\C_{\D}$ at least one derivative of $\D$ falls on~$J_0 a$ or~$a$. 
Therefore, using \eqref{EQ76}, \eqref{EQ81}, \eqref{EQ88}, and \eqref{EQ96}, we obtain
\begin{align}
	\sum_{\D \in \{\bp, \pt, \pt\bp\}}
	\int_{\tau}^{t}
	\C_{\D}
	\,ds
	\les 
	\mathcal{H} O(X).
	\label{EQ135}
\end{align}
For the case $\D= \pa_{tt}$, the term $\C_{\pa_{tt},1}$ can be estimated similarly to~\eqref{EQ135} using~\eqref{EQ128}. 
For the term $\C_{\pa_{tt},2}$, we integrate by parts in time to obtain
\begin{align}
	\begin{split}
	&\int_{\tau}^{t}
	\C_{\pa_{tt},2}
	\,ds
	=
	\int_{\tau}^{t}
	\int_{\Omega}
	(J_0 a_{ki} \pa_k \pa_{tt} v_i -
	\pa_{tt} (J_0 a_{ki} \pa_k v_i)
	)	
	\pa_{tt} q
	\,ds
	\\&\indeq
	=
	\left.
	\int_{\Omega}
	(J_0 a_{ki} \pa_k \pa_{tt} v_i -
	\pa_{tt} (J_0 a_{ki} \pa_k v_i)
	)	
	\pa_{t} q
	\right|_{\tau}^t
	-
	\int_{\tau}^{t}
	\int_{\Omega}
	\pt
	(J_0 a_{ki} \pa_k \pa_{tt} v_i -
	\pa_{tt} ( J_0  a_{ki} \pa_k v_i)
	)	
	\pa_{t} q
	\,ds
	\\&\indeq
 	=: K_1+K_2.
 	\llabel{EQ136}
	\end{split}
\end{align}
For the term $K_1$, we have
\begin{align}
	\begin{split}
	K_1
	&=
	-
	2\left.
	\int_{\Omega}
	J_0 
	\pa_{t} a_{ki} \pt \pa_k v_i \pt q
	\right|^t_{\tau}
	-	
	\left.
	\int_{\Omega}
	J_0 
	\pa_{tt} a_{ki}  \pa_k v_i \pt q
	\right|^t_{\tau}
	\les \mathcal{H} O(X),
	\label{EQ137}
	\end{split}
\end{align} 
where we used \eqref{EQ81} and~\eqref{EQ128}.
For the term $K_{2}$, all the terms can be estimated in a similar fashion as in \eqref{EQ137}, except for the terms when $v$ receive two time derivatives and one spatial derivative.
For such terms, we use the Young's inequality to get
\begin{align}
	\begin{split}
	\left|
	\int_{\tau}^t
	\int_{\Omega} \nabla \pa_{tt}
	v \nabla v \pa_t q
	\,ds
	\right|	
	\les
	\delta \int_{\tau}^{t}
	\DD_{\pa_{tt}}
	\,ds
	+
	C_\delta
	\mathcal{H} O(X).
	\label{EQ138}
	\end{split}
\end{align}
For the term $\C_{\pa_{tt},3}$, we proceed analogously as in \eqref{EQ138} to obtain the same bound.
For the case when $\D = \bp^2$, we use the same strategy to bound the terms in $\C_{\bp^2}$, except for when two tangential derivatives are falling on $a$ in first term in \eqref{EQ74}, since we have no control over~$\Vert a\Vert_{H^3}$. 
Due to the Piola identity \eqref{EQ31}, that term vanishes. 
Therefore, we get
\begin{align}
	\int_{\tau}^{t}
	\C_{\bp^2}
	\,ds
	\les 
	\mathcal{H} O(X).
	\llabel{EQ139}
\end{align}
The proof of \eqref{EQ107} is then concluded.
\colb

\appendix
\section{An ODE-type lemma}
Here we recall an ODE inequality from \cite{KO1} needed in the
conclusion of the proof of the main theorem.

\colb
\begin{Lemma}
	\cite{KO1}
	\label{LODE}
	Given constants $C\geq 1$, $\gamma \in (0,1]$, and $\lambda \in (0,1/500C^2]$, there exists $\epsilon>0$ with the following property.
	Suppose $f\colon [0,\infty) \to [0,\infty)$ is a continuous function satisfying
	\begin{align}
		f(t)+\lambda \int_{\tau}^{t} f(s)\,ds
		\leq 
		C (1+\lambda^2 (t-\tau)) f(\tau)
		+ 
		C\lambda^2 \int_{\tau}^{t} f(s)\,ds
		+ h(t) O(f),
		\llabel{EQ140}
	\end{align}
for any $t>0$ and $\tau \in [0,t]$, provided that
	\begin{align}
		h(t):=
		\sup_{s\in [0,t]} 
		\left(
		f^{1/2}(s)
		+ 
		\int_{0}^{s} f^{1/2}
		\right)
		\leq \gamma,
		\label{EQ141}
	\end{align}
	where $O(f)$ denotes any term involving powers of $\lambda$, $C$, and $(t-\tau)$, and at least two factors chosen from $\{f(\tau)^{1/2}, f(t)^{1/2}, \int_{\tau}^{t} f^{1/2}(s)\,ds \}$ or one factor in the form $\int_{\tau}^t f(s)\,ds$.
	Then, the condition $f(0)\leq \epsilon$ implies \eqref{EQ141} and
	\begin{align}
		f(t) 
		\leq 
		30C f(0) e^{-t \lambda/2C}
		\comma t\geq 0.
		\llabel{EQ142}
	\end{align}
\end{Lemma}
\colb

\colb
\section*{Acknowledgments}
MB and BM were supported by the Croatian Science Foundation under project IP-2022-10-2962, while
IK was supported in part by the NSF grant DMS-2205493 and Simons grant SFI-MPS-TSM-00025796. BM was also supported by: the European Union Next Generation EU through the National Recovery and Resilience Plan 2021--2026; an institutional grant from the University of Zagreb, Faculty of Science: project IK IA 1.1.3. Impact4Math. MB was also partly supported by the NextGenerationEU framework through the project "DEEPWAVE" at the University of Zagreb Faculty of Electrical Engineering and Computing.

\colb

\end{document}